\documentclass[a4paper,10pt,leqno]{amsart}
\usepackage[T1]{fontenc}
\usepackage[utf8]{inputenc}
\usepackage[margin=1.2in]{geometry}
\usepackage{lmodern}
\usepackage[english]{babel}
\usepackage{graphicx}
\usepackage{amsmath}
\usepackage{amsfonts}
\usepackage{amssymb} 
\usepackage{amsthm}
\usepackage{hyperref}
\usepackage{bbm}
\usepackage{bm} 
\usepackage[foot]{amsaddr}
\usepackage{lipsum}
\usepackage{xcolor}
\usepackage{url}
\usepackage{subfig}
\usepackage{comment}
\theoremstyle{plain}
\newtheorem{thm}{Theorem}
\newtheorem{prop}{Proposition}

\newtheorem{rem}{Remark}

\def \be {\begin{equation}}
\def \ee {\end{equation}}

\begin{document}
\title[Oscillatory behavior in a model of mean field interacting spins]{Oscillatory behavior in a model of non-Markovian mean field interacting spins}

\author{Paolo Dai Pra}
\author{Marco Formentin}
\author{Guglielmo Pelino}
\noindent \address[P. Dai Pra, M. Formentin, G. Pelino]{Department of Mathematics ``Tullio Levi-Civita'', \newline \indent University of Padua, \newline \indent Via Trieste 63, 35121 Padova, Italy.}
\address[P. Dai Pra]{Department of Computer Science, \newline \indent University of Verona, \newline \indent Strada Le Grazie 15, 37134 Verona, Italy.}
\address[M. Formentin]{Padova Neuroscience Center, \newline \indent University of Padua,\newline \indent via Giuseppe Orus 2, 35131 Padova, Italy.}
\address[G. Pelino]{School of Mathematics, \newline \indent University of Edinburgh, James Clerk Maxwell Building, \newline \indent Mayfield Rd, EH9 3FD, Edinburgh, UK.}
\vspace{1cm}
\email[P. Dai Pra]{daipra@math.unipd.it}
\email[M. Formentin]{marco.formentin@unipd.it}
\email[G. Pelino]{guglielmo.pelino@math.unipd.it}
\thanks{\textbf{Acknowledgments:} The authors acknowledge financial support through the project ``Large Scale Random Structures'' of the Italian Ministry of Education, Universities and Research (PRIN 20155PAWZB-004). The last author is partially supported by the PhD Program in Mathematical Science, Department of Mathematics, University of Padua (Italy), Progetto Dottorati - Fondazione Cassa di Risparmio di Padova e Rovigo. We would finally like to thank Giambattista Giacomin for helpful discussions.
}


\date{\today}

\begin{abstract}
We analyze a non-Markovian mean field interacting spin system, related to the Curie--Weiss model. 
We relax the Markovianity assumption by replacing the memoryless distribution of the waiting times of a classical spin-flip dynamics with a distribution with memory. The resulting stochastic evolution for a single particle is a spin-valued renewal process, an example of two-state semi-Markov process. We associate to the individual dynamics an equivalent Markovian description, which is the subject of our analysis. We study a corresponding interacting particle system, where a mean field interaction is introduced as a time scaling, depending on the overall magnetization of the system, on the waiting times between two successive particle's jumps. Via linearization arguments on the Fokker-Planck mean field limit equation, we give evidence of emerging periodic behavior. Specifically, numerical analysis on the discrete spectrum of the linearized operator, characterized by the zeros of an explicit holomorphic function, suggests the presence of a Hopf bifurcation for a critical value of the temperature, which is in accordance with the one obtained by simulating the $N$-particle system.
\end{abstract}

\maketitle

\section{Introduction}

Emerging periodic behavior in complex systems with a large number of interacting units is a commonly observed phenomenon in neuroscience 
(\cite{ermentrout}), ecology (\cite{turchin}), socioeconomics (\cite{chen, weidlich}) and life sciences in general. From a mathematical standpoint, when modeling such a phenomenon it is natural to consider large families of microscopic identical units evolving through noisy interacting dynamics where  each individual particle has no natural tendency to behave periodically and oscillations are rather an effect of self-organization as they emerge in the macroscopic limit when the number of particles tends to infinity.  Within this modeling framework mean field models have received much attention due to their analytical tractability. Throughout the paper we refer to the emergence of self-organized periodic oscillations with the term self-sustained periodic behavior. One of the goals of the mathematical theory in this field is to understand which types of microscopic interactions and mechanisms can lead to or enhance the above self-organization. Among others, we cite noise (\cite{daipra_regoli}, \cite{scheutzow}, \cite{touboul2}), dissipation in the interaction potential (\cite{andreis}, \cite{collet_dp_forme},  \cite{collet_formentin}, \cite{daipra_reg_diss}), delay in the transmission of information and/or frustration in the interaction network (\cite{collet_form_tov}, \cite{locherbach}, \cite{touboul}). In particular, in \cite{locherbach} the authors consider non-Markovian dynamics, studying systems of interacting nonlinear Hawkes processes for modeling neurons.

Although not proved in general, a strong belief in the literature is that, at least for Markovian dynamics, self-sustained periodic behavior cannot emerge if one does not introduce some time-irreversible phenomenon in the dynamics, as it is the case in all the above cited works (see e.g.\!  \cite{bertini}, \cite{giacomin_poquet}). The model treated here, in which the limit dynamics is still reversible with respect to the stationary distribution around which cycles emerge (see Remark \ref{reversibility_renewal} below), suggests that this paradigm could be false for the non-Markovian case.

Specifically, we give numerical and mathematical evidence of the emergence of self-sustained periodic behavior in a mean field spin system related to the Curie--Weiss model, which happens to belong to the following universality class: it features the presence of a unique stable neutral phase for values of the parameters corresponding to high temperatures, the emergence of periodic orbits in an intermediate range of the parameter values, and a subsequent ferromagnetic ordered phase for increasingly lower temperatures.  Our recipe consists in replacing the Poisson distribution of the spin-flip times with another renewal process, thus making the individual spin dynamics non-Markovian. In details, we consider the distribution of the interarrival times to have tails proportional to $e^{-t^{\gamma +1}}$, for $\gamma =1,2$. Then, we introduce an interaction among the spins via a time rescaling depending on the overall magnetization of the system.
The specific choice of interarrival time distribution makes the computations developed in Section \ref{proofs} the easiest as possible (to our knowledge), allowing for an explicit characterization of the discrete spectrum of the linearized operator. A question which can arise naturally is whether similar results can be found for other classes of waiting times. Although we do not have a general answer to this, we want to remark that simulations with different distributions highlighted the same characteristics (e.g.\! Gamma distribution, tails proportional to $e^{-t^{\gamma +1}}$ with $\gamma \in \mathbb{R}$ such that $\gamma \geq 1$). All the working examples we considered feature exponentially or super-exponentially decaying tails. On the other hand, we have examples of polynomial tails (e.g.\! inverse Gamma distribution) where no oscillatory behavior was experienced.

The paper is organized as follows: in Section \ref{mf_model} we describe the model and the results obtained. In particular, before introducing the model (Subsection \ref{meanfield_model}), we start by recalling basic facts about the Curie--Weiss model and its phase transitions (Subsection \ref{motivation}); we then proceed with the results on the propagation of chaos (Subsection \ref{prop_chaos}), and on the linearized Fokker-Planck equation around a neutral equilibrium, for two different choices of renewal dynamics (Subsection \ref{local_fp}). Notably, we determine the discrete spectrum of the linearized operator in terms of the zeros of two holomorphic functions. Section \ref{numerics} contains the numerical results on the discrete spectrum, studied as a function of the interaction parameters (Subsection \ref{numerical_evidence_eigen}). These results are then compared in Subsection \ref{finite_particle} with the ones obtained by simulating the finite particle system, finding a precise accordance between the two approaches. Section \ref{proofs} contains the proofs of the results of Section \ref{mf_model}.

\section{Model and Results}
\label{mf_model}

\subsection{Motivation}
\label{motivation}
As we mentioned above, the model we consider can be seen as a proper modification of the Curie--Weiss dynamics. When we refer to the latter, we mean a spin-flip type Markovian dynamics for a system of $N$ interacting spins $\sigma_i \in \left\{-1,1\right\}$, $i =1,\dots,N$, which is \textit{reversible} with respect to the equilibrium Gibbs probability measure on the space of configurations $\left\{-1,1\right\}^N$, 
\begin{equation}
\label{eqn:gibbs_cw}
P_{N,\beta}(\bm{\sigma}) := \frac{1}{Z_N(\beta)} \exp\left[-\beta H(\bm{\sigma})\right],
\end{equation}
with $\bm{\sigma} := (\sigma_1,\dots,\sigma_N) \in \left\{-1,1\right\}^N$, $\beta > 0$ (\textit{ferromagnetic} case), $Z_N(\beta)$ is a normalizing constant, and $H$ is the \textit{Hamiltonian}, which in the Curie--Weiss setting is given by
\begin{equation}
\label{eqn:hamilt}
H_N(\bm{\sigma}) := - \frac{1}{2N}\left(\sum_{i=1}^N \sigma_i\right)^2.
\end{equation}
Denote also the empirical \textit{magnetization} as $m^N := \frac{1}{N}\sum_{i=1}^N \sigma_i$.
Note that the distribution \eqref{eqn:gibbs_cw} gives higher probability to the configurations with minimal energy, which by \eqref{eqn:hamilt} are the ones where the individual spins are aligned in the same state. The equilibrium model undergoes a \textit{phase transition} tuned by the interaction parameter $\beta >0$, which can be recognized by proving a Law of Large Numbers for the equilibrium empirical magnetization
\begin{equation}
\label{eqn:phase_tr_eq}
\text{Law}(m^N) \xrightarrow{N \to +\infty}
\begin{cases}
\delta_0, \ \ \ &\text{ if } \beta \leq 1,\\
\frac{1}{2}\delta_{+m_\beta} + \frac{1}{2}\delta_{-m_\beta}, \ \ \ &\text{ if } \beta > 1,
\end{cases}
\end{equation}
where $m_\beta > 0$ is the so-called \textit{spontaneous magnetization}.
When we turn to the dynamics, different choices can be made in order to satisfy the above-mentioned reversibility with respect to \eqref{eqn:gibbs_cw}. The prototype is a continuous-time spin-flip dynamics defined in terms of the infinitesimal generator $L$, applied to a function $f : \left\{-1,1\right\}^N \to \mathbb{R}$,
\begin{equation}
\label{eqn:cw_gen}
Lf(\bm{\sigma}) = \sum_{i=1}^N e^{-\beta \sigma_i m^N}\left[ f(\bm{\sigma}^i)-f(\bm{\sigma})\right],
\end{equation}
where $\bm{\sigma}^i \in \left\{-1,1\right\}^N$ is obtained from $\bm{\sigma}$ by \textit{flipping} the $i$-th spin.
Dynamics \eqref{eqn:cw_gen} induces a continuous-time Markovian evolution for the empirical magnetization process $m^N(t)$, which is given in terms of a generator $\mathcal{L}$ applied to a function $g :[-1,1] \to \mathbb{R}$:
\begin{equation}
\label{eqn:cw_real_gen}
\mathcal{L}^Ng(m) = N \frac{1+ m}{2} e^{-\beta m}\left[g\left(m-\frac{2}{N}\right)-g(m)\right] + N \frac{1-m}{2}e^{\beta m}\left[g\left(m+\frac{2}{N}\right) - g(m)\right].
\end{equation}
It is easy to obtain the weak limit of the sequence of processes $\big(m^N(t)\big)_{t \geq 0}$, by studying the uniform convergence of the generator \eqref{eqn:cw_real_gen} as $N \to +\infty$ (see e.g.\! \cite{ethier}).
The limit process $(m(t))_{t \geq 0}$ is deterministic and solves the Curie--Weiss ODE
\begin{equation}
\label{eqn:dynamics_cw}
\begin{cases}
\dot{m}(t) = 2\sinh(\beta m(t)) - 2 m(t) \cosh(\beta m(t)),\\
m(0) = m_0 \in [-1,1].
\end{cases}
\end{equation}
The presence of the phase transition highlighted in \eqref{eqn:phase_tr_eq} can be recognized as well in the out-of-equilibrium dynamical model \eqref{eqn:dynamics_cw}.
Indeed, studying the long-term behavior of \eqref{eqn:dynamics_cw}, one finds that:
\begin{itemize}
\item for $\beta \leq 1$, \eqref{eqn:dynamics_cw} possesses a unique stationary solution, globally attractive, constantly equal to $0$;
\item for $\beta > 1$, $0$ is still stationary but it is unstable; two other symmetric stationary locally attractive solutions, $\pm m_\beta$, appear: the two non-zero solutions to $m = \tanh(\beta m)$.
The dynamics $m(t)$ gets attracted for $t \to +\infty$ to the polarized stationary state which has the same sign as the initial magnetization $m_0$.
\end{itemize}

Another concept which we refer to in what follows is that of a \textit{renewal} process, a generalization of the Poisson process. We identify a renewal process with the sequence of its \textit{interarrival} times (also commonly referred to as sojourn times or waiting times in the literature) $\left\{T_n\right\}_{n=1}^\infty$, i.e.\! the holding times between the occurrences of two consecutive events. The Poisson process is characterized by having independent and identically distributed interarrival times, where each $T_i$ is exponentially distributed. In particular, the \textit{memoryless} property $\mathbb{P}(T_i > s+t | T_i > t) = \mathbb{P}(T_i > s)$, holds for any $s,t \geq 0$.
The interarrival times of a renewal process are still independent and identically distributed, but their distribution is not required to be exponential. We recall that a continuous-time homogeneous Markov chain can be identified by a Poisson process, modeling the jump times, and a stochastic transition matrix, identifying the possible arrival states at each jump time.
Due to the lack of the memoryless property, when one replaces the Poisson process in the definition of the spin-flip dynamics with a more general renewal process, the resulting evolution is thus non-Markovian. In the literature, the associated dynamics is referred to as \textit{semi-Markov} process, first introduced by Levy in \cite{levy}.

\subsection{The dynamics}
\label{meanfield_model}
In order to introduce the model, we start by observing that the Curie--Weiss dynamics \eqref{eqn:cw_gen}, as any spin-flip Glauber dynamics, can be obtained by adding interaction to a system of independent spin-flips: at the times of a Poisson process of intensity $1$, the spin in a given site flips; different sites have independent Poisson processes. Our aim here is to replace Poisson processes by more general renewal processes, otherwise keeping the structure of the interaction.
For the moment we focus on a single spin $\sigma(t) \in \left\{-1,1\right\}$. If driven by a Poisson process of intensity $1$, its dynamics has infinitesimal generator
\begin{equation}
\label{eqn:markovian}
\mathcal{L}f(\sigma) = f(-\sigma) - f(\sigma),
\end{equation}
$f : \left\{-1,1\right\} \to \mathbb{R}$.
If the Poisson process is replaced by a renewal process, the spin dynamics is not Markovian. In what follows, we refer to the resulting dynamics as a spin-valued renewal process, that is an example of two-states semi-Markov process.
We can associate a Markovian description to the latter: define $y(t)$ as the time elapsed since the last spin-flip occured up to time $t$. Suppose that the waiting times $\tau$ (interchangeably referred to as interarrival times) of the renewal satisfy
\begin{equation}
\mathbb{P}(\tau > t) = \varphi(t),
\end{equation}
for some smooth function $\varphi : [0,+\infty) \to \mathbb{R}$. Then, the pair $(\sigma(t),y(t))_{t \geq 0}$ is Markovian with generator
\begin{equation}
\label{eqn:markovianized_gen}
\mathcal{L}f(\sigma,y) = \frac{\partial f}{\partial y}(\sigma,y) + F(y)[f(-\sigma,0) - f(\sigma,y)],
\end{equation}
for $f :  \left\{-1,1\right\}  \times \mathbb{R}^+ \to \mathbb{R}$, with
\begin{equation}
\label{eqn:form_of_the_rate}
 F(y) := -\frac{\varphi'(y)}{\varphi(y)}.
\end{equation}
This is equivalent to say that the couple $(\sigma(t), y(t))_{t \geq 0}$ evolves according to
\begin{equation}
\begin{cases}
\label{eqn:dyn}
(\sigma(t), y(t))  \mapsto (-\sigma(t),0),  \ \ \ \text{    with rate } F(y(t)), \\
 dy(t) = dt, \ \ \ \text{    otherwise.} 
\end{cases}
\end{equation}
Expression \eqref{eqn:form_of_the_rate} for the jump rate follows by observing that, for an interarrival time $\tau$ of the jump process $\sigma(t)$, we have
$$
\mathbb{P}(\sigma(t+h) = -\sigma | \sigma(t) = \sigma) = 1 - \mathbb{P}(\tau > t+h | \tau > t) = 1 - \frac{\varphi(t+h)}{\varphi(t)},
$$
for any $h > 0$.
Observe that when the $\tau$'s are exponentially distributed $F(y) \equiv 1$, so we get back to dynamics \eqref{eqn:markovian}.
Dynamics \eqref{eqn:markovianized_gen} can be perturbed by allowing the distribution of the waiting time for a spin-flip to depend on the current spin value $\sigma$; the simplest way is to model this dependence as a time scaling:
\begin{equation}
\label{eqn:tau_x}
\mathbb{P}(\tau > t | \sigma) = \varphi(a(\sigma) t).
\end{equation}
Under this distribution for the waiting times the generator of $(\sigma(t),y(t))_{t \geq 0}$ becomes:
$$
\mathcal{L}f(\sigma,y) = \frac{\partial f}{\partial y}(\sigma,y) + a(\sigma) F(a(\sigma)y)[f(-\sigma,0) - f(\sigma,y)].
$$

On the basis of what seen above, it is rather simple to define a system of mean-field interacting spins with non-exponential waiting times. 
For a collection of $N$ pairs $(\sigma_i(t),y_i(t))_{i=1,\dots,N}$, we set $m^N(t) := \frac{1}{N}\sum_{i=1}^N \sigma_i(t)$ to be the magnetization of the system at time $t$, and a parameter $\beta >0$ tuning the interaction between the particles.
The interacting dynamics is 
\begin{equation}
\label{eqn:int_dyn}
\begin{cases}
(\sigma_i(t), y_i(t)) \mapsto (-\sigma_i(t),0),  \ \ \ \text{    with rate } \ \  F\left(y_i(t) e^{-\beta \sigma_i(t) m^N(t)}\right)e^{-\beta \sigma_i(t) m^N(t)}, \\
dy_i(t) = dt, \ \ \ \text{    otherwise.} 
\end{cases}
\end{equation}
Denoting $\bm{\sigma}:= (\sigma_1,\dots, \sigma_N) \in \left\{-1,1\right\}^N$, $\bm{y} := (y_1,\dots,y_N) \in (\mathbb{R}^+)^N$, $m^N := \frac{1}{N}\sum_{i=1}^N \sigma_i$, the associated infinitesimal generator is
\begin{equation}
\label{eqn:mf_renewal}
\mathcal{L}^N f(\bm{\sigma},\bm{y}) = \sum_{i=1}^N \frac{\partial f}{\partial y_i}(\bm{\sigma},\bm{y}) + \sum_{i=1}^N F\left(y_i e^{-\beta \sigma_i m^N}\right)e^{-\beta \sigma_i m^N}\left[f(\bm{\sigma}^i,\bm{y}^i) - f(\bm{\sigma},\bm{y})\right],
\end{equation}
where $\bm{\sigma}^i$ is obtained from $\bm{\sigma}$ by flipping the $i$-th spin, while $\bm{y}^i$ by setting to zero the $i$-th coordinate. The additional factor $e^{-\beta \sigma_i(t) m^N(t)}$ in the jump rate in \eqref{eqn:int_dyn} follows from the observation we made in \eqref{eqn:tau_x} and the definition of $F(y) = -\frac{\varphi'(y)}{\varphi(y)}$. Note that, for $F \equiv 1$, we retrieve the Curie--Weiss dynamics \eqref{eqn:cw_gen} for the spins.

\subsection{Propagation of chaos}
\label{prop_chaos}
The macroscopic limit and propagation of chaos for the above class of models should be standard, although some difficulties may arise for general choices of $F$ not globally Lipschitz. 
For computational reasons which will be made clear below, we focus on the case $F(y) = y^\gamma$, for $\gamma \in \mathbb{N}$, which corresponds to considering, in the single spin model, the tails of the distribution of the interarrival times to be $\varphi(t) \propto e^{-\frac{t^{\gamma+1}}{\gamma+1}}$. 

When $F(y) = y^\gamma$, \eqref{eqn:int_dyn} becomes 
\begin{equation}
\label{eqn:final_gamma_ren}
\begin{cases}
(\sigma_i(t), y_i(t)) \mapsto (-\sigma_i(t),0),  \ \ \ \text{    with rate } \ \  y_i^{\gamma}(t) e^{-(\gamma+1)\beta \sigma_i(t) m^N(t)}, \\
dy_i(t) = dt, \ \ \ \text{    otherwise.} 
\end{cases}
\end{equation}
As for the Curie--Weiss model, dynamics \eqref{eqn:final_gamma_ren} is subject to a cooperative-type interaction: the spin-flip rate is larger for particles which are not aligned with the majority. 
Assuming propagation of chaos, at the macroscopic limit $N \to +\infty$ the representative particle $(\sigma(t), y(t))$ has a mean-field dynamics
\begin{equation}
\label{eqn:mf_final_renewal}
\begin{cases}
(\sigma(t), y(t)) \mapsto (-\sigma(t),0),  \ \ \ \text{    with rate } \ \ y^{\gamma}(t) e^{-(\gamma+1)\beta \sigma(t) m(t)}, \\
dy(t) = dt, \ \ \ \text{    otherwise,} 
\end{cases}
\end{equation}
with $m(t) = \mathbb{E}[\sigma(t)]$. To this dynamics we can associate (see \cite{kolokoltsov_book}) the non-linear infinitesimal generator
\begin{equation}
\label{eqn:mf_limit_renewal}
\mathcal{L}(m(t)) f(\sigma,y) = \frac{\partial f}{\partial y}(\sigma,y) + y^\gamma e^{-(\gamma+1)\beta \sigma m(t)}\left[f(-\sigma,0) - f(\sigma,y)\right],
\end{equation}
where the non-linearity is due to the dependence of the generator on $m(t)$, a function of the joint law at time $t$ of the processes $(\sigma(t),y(t))$.
In Section \ref{proofs} we study rigorously the well-posedeness of the pre-limit and limit dynamics and the propagation of chaos. The main result is collected in the following
\begin{thm}[Propagation of chaos]
\label{prop_chaos_renewal}
Fix $\gamma \in \mathbb{N}$, and let $T>0$ be the final time in \eqref{eqn:final_gamma_ren} and \eqref{eqn:mf_final_renewal}. Assume that $(\sigma_i(0), y_i(0))_{i=1,\dots,N}$ are $\mu_0$-chaotic for some probability distribution $\mu_0$ on $\left\{-1,1\right\} \times \mathbb{R}^+$. Then, the sequence of empirical measures $(\mu_t^N)_{t \in [0,T]}$ converges in distribution (in the sense of weak convergence of probability measures) to the deterministic law $(\mu_t)_{t \in [0,T]}$ on the path space of the unique solution to Eq.\! \eqref{eqn:mf_final_renewal} with initial distribution $\mu_0$.
\end{thm}

\subsection{Local analysis of the Fokker-Planck}
\label{local_fp}
In this section we illustrate the results on the local analysis of the Fokker-Planck equation for the mean-field limit dynamics \eqref{eqn:mf_final_renewal} with $\gamma = 1$ and $\gamma = 2$. Our approach is the following: we find a neutral stationary solution of interest, we linearize formally the dynamics around that equilibrium and we compute the discrete spectrum of the associated linearized operator, which we show to be given by the zeros of an explicit holomorphic function $H_{\beta,\gamma}(\lambda)$. In Subsection \ref{numerical_evidence_eigen} we then study numerically the character of the eigenvalues when $\beta$ varies: for both $\gamma=1,2$, we find that for all $\beta < \beta_c(\gamma)$ all eigenvalues have negative real part; at $\beta_c(\gamma)$ two eigenvalues are conjugate and purely imaginary, suggesting the possible presence of a Hopf bifurcation in the limit dynamics. These critical values of $\beta$ are then compared to the ones obtained by simulating the finite particle system in Subsection \ref{finite_particle}.
   
The Fokker-Planck equation associated to \eqref{eqn:mf_final_renewal} is a PDE describing the time evolution of the density function $f(t,\sigma,y)$ of the limit process $(\sigma(t),y(t))$. It is given by
\begin{equation}
\label{eqn:kfp_renewal}
\begin{cases}
\frac{\partial}{\partial t} f(t,\sigma,y) + \frac{\partial}{\partial y} f(t,\sigma,y) + y^\gamma e^{-(\gamma + 1)\beta \sigma m(t)} f(t,\sigma,y) = 0,\\
f(t,\sigma,0) = \int_{0}^{+\infty} y^\gamma e^{(\gamma + 1)\beta \sigma m(t)} f(t,-\sigma,y) dy, \\
m(t) =  \int_{0}^{\infty} [f(t,1,y) - f(t,-1,y)]dy,\\
1 = \int_{0}^{\infty} [f(t,1,y) + f(t,-1,y)]dy,\\
f(0,\sigma,y) = f_0(\sigma,y), \text{ for } \sigma\in \left\{-1,1\right\}, \ y \in \mathbb{R}^+.
\end{cases}
\end{equation}
 A general study of \eqref{eqn:kfp_renewal} is beyond the scope of this work. Here we just observe that \eqref{eqn:kfp_renewal} can be seen as a system of two quasilinear PDEs (one for $\sigma = 1$ and another for $\sigma = -1$), where the non-linearity enters in an integral form through $m(t)$ in the exponent of the rate function. Moreover, the boundary integral condition in the second line poses additional challenges.

Nevertheless, it is easy to exhibit a particular stationary solution to \eqref{eqn:kfp_renewal}:
\begin{prop}
\label{neutral_eq}
The function
\begin{equation}
\label{eqn:stat_sol}
f^*(\sigma,y) = \frac{1}{2\Lambda}e^{-\frac{y^{\gamma+1}}{\gamma+1}},
\end{equation}
with $\Lambda := \int_{0}^{+\infty} e^{-\frac{y^{\gamma+1}}{\gamma+1}}$, is a stationary solution to Sys.\! \eqref{eqn:kfp_renewal} with $m = 0$.
\end{prop}

\begin{rem}
\label{reversibility_renewal}
Let $g^*(\sigma)$ be the marginal of $f^*(\sigma,y)$ with respect to the first coordinate. Then, $g^*(\sigma)$ is a stationary reversible distribution for the limit renewal process $(\sigma(t))_{t \geq 0}$. Indeed, by choosing $\sigma(0) \sim g^*$, $g^*(1) = g^*(-1) = \frac{1}{2}$, we have that $m(t) \equiv 0$ and $(\sigma(t))_{t \geq 0}$ is a renewal process with interarrival times $\tau$ such that $\mathbb{P}(\tau > t) \propto e^{-\frac{t^{\gamma+1}}{\gamma+1}}$ independently of the value of $\sigma$, so its law is invariant by time reversal.
\end{rem}

The linearization of the operator associated to Sys.\! \eqref{eqn:kfp_renewal} around the neutral equilibrium \eqref{eqn:stat_sol}, yields the following eigen-system
\begin{equation}
\label{eqn:linear_renewal}
\begin{cases}
\frac{\partial}{\partial y} g(\sigma,y) + y^\gamma g(\sigma,y) - \frac{\beta \sigma k(\gamma + 1)}{2\Lambda}y^\gamma  e^{-\frac{y^{\gamma +1}}{\gamma +1}} = -\lambda g(\sigma,y),\\
g(\sigma,0) = \frac{\beta \sigma k (\gamma + 1)}{2\Lambda} + \int_0^\infty g(-\sigma,y) y^\gamma dy,\\
\int_0^\infty [g(\sigma,y) + g(-\sigma,y)]dy = 0, \ \ \ (\sigma,y) \in \left\{-1,1\right\} \times \mathbb{R}^+,
\end{cases}
\end{equation}
where $k = 2\int_0^\infty g(1,y) dy$, and $\Lambda = \int_0^\infty e^{-\frac{y^{\gamma +1}}{\gamma + 1}} dy$. For a formal derivation see Subsection \ref{linearized_stationary}.
We work out the computations of the discrete spectrum of the linearized operator for the two cases $\gamma = 1$, $\gamma = 2$. 

\subsubsection{Case $\gamma = 1$}
In this case, $\Lambda = \sqrt{\frac{\pi}{2}}$, and the eigen-system \eqref{eqn:linear_renewal} becomes
\begin{equation}
\label{eqn:gamma_1}
\begin{cases}
\frac{\partial}{\partial y} g(\sigma,y) + y g(\sigma,y) + \lambda g(\sigma,y)  = \beta \sigma k \left(\sqrt{\frac{\pi}{2}}\right)^{-1}y e^{-\frac{y^2}{2}} \\
g(\sigma,0) = \beta \sigma k \left(\sqrt{\frac{\pi}{2}}\right)^{-1} + \int_0^{\infty} y g(-\sigma,y)dy,\\
\int_0^{\infty} [g(\sigma,y) + g(-\sigma,y)]dy = 0,
\end{cases}
\end{equation}
where $k = 2\int_0^\infty g(1,y)dy$. 
\begin{prop}
\label{gamma1_prop}
The solutions in $\lambda \in \mathbb{C}$ to \eqref{eqn:gamma_1} are the zeros of the holomorphic function
\begin{equation}
\label{eqn:holo_g1}
H_{\beta,1}(\lambda) := H_1(\lambda)\left[-4\beta - \lambda^3 \sqrt{\frac{\pi}{2}}\right] + \sqrt{2\pi}\lambda^2 - 4\beta \lambda + 2\beta \sqrt{2\pi},
\end{equation}
with 
\begin{equation}
\label{eqn:expression_h1g1}
H_1(\lambda) := \int_0^\infty e^{-\frac{y^2}{2}}e^{-\lambda y}.
\end{equation}
Moreover, it holds
\begin{equation}
\label{eqn:h1_g1}
H_1(\lambda) = \sqrt{\frac{\pi}{2}}\sum_{m=0}^\infty \frac{\lambda^{2m}}{(2m)!!} - \lambda \sum_{m=0}^\infty \frac{(2\lambda)^{2m}m!}{(2m + 1)!}\frac{1}{2^m}.
\end{equation}
\end{prop}

\subsubsection{Case $\gamma = 2$}
In this case the eigen-system is given by
\begin{equation}
\label{eqn:gamma_2}
\begin{cases}
\frac{\partial}{\partial y} g(\sigma,y) + y^2 g(\sigma,y) +\lambda g(\sigma,y) =  \frac{3}{2\Lambda}\beta \sigma k y^2 e^{-\frac{y^3}{3}} ,\\
g(\sigma,0) = \frac{3}{2\Lambda} \beta \sigma k  + \int_0^{\infty} y^2 g(-\sigma,y)dy,\\
\int_0^{\infty} [g(\sigma,y) + g(-\sigma,y)]dy = 0,
\end{cases}
\end{equation}
where $\Lambda = \int_0^\infty e^{-\frac{y^3}{3}} = \frac{\Gamma\left(\frac{1}{3}\right)}{3^{2/3}}$, $k = 2 \int_0^{\infty} g(1,y) dy$, and $\Gamma(\cdot)$ is the Gamma function.
\begin{prop}
\label{gamma2_prop}
The solutions in $\lambda \in \mathbb{C}$ to \eqref{eqn:gamma_2} are the zeros of the holomorphic function
\begin{equation}
\label{eqn:holo_g2}
\begin{aligned}
H_{\beta,2}&(\lambda):= H_2(\lambda)\Bigg[12\beta - \lambda^4 \Lambda + 6\beta \lambda \Lambda - 6\beta \lambda 3^{1/3}\Gamma(4/3) + 3\beta \lambda^2 3^{2/3}\Gamma(5/3)\\
&- 6\beta\lambda^2\frac{\Gamma(2/3)}{3^{1/3}}\Bigg]+\Bigg[2\Lambda \lambda^3 - 12\beta\Lambda + 12\beta\frac{\Gamma(2/3)}{3^{1/3}}\lambda - 6\beta \lambda^2 \Bigg],
\end{aligned}
\end{equation}
with 
\begin{equation}
\label{eqn:expression_h2g2}
H_2(\lambda) := \int_0^\infty e^{-\lambda y} e^{-\frac{y^3}{3}} dy.
\end{equation}
Moreover, it holds
\begin{equation}
\label{eqn:h2_g2}
H_2(\lambda) = \sum_{n=0}^{\infty} (-1)^n \frac{\lambda^n}{n!} 3^{\frac{1}{3}(n-2)}\Gamma\left(\frac{n+1}{3}\right).
\end{equation}
\end{prop}

\section{Numerical results}
\label{numerics}
\subsection{Numerical evidence on the eigenvalues}
\label{numerical_evidence_eigen}
We studied numerically the two eigenvalues equations 
\begin{equation}
\label{eqn:g1_eigenvalues}
H_{\beta,1}(\lambda) = 0,
\end{equation}
and 
\begin{equation}
\label{eqn:g2_eigenvalues}
H_{\beta,2}(\lambda) = 0.
\end{equation}
We used a numerical root finding built-in function of the software \texttt{Mathematica}, specifically \texttt{FindRoot}, starting the search from different initial points of the complex plane and from different values of $\beta$. Here we report the results:
\vspace{0.2cm}

\begin{itemize}
\item Case $\gamma = 1$:
\vspace{0.15cm}
 \begin{itemize}
 \item[(1.1)] we find two conjugate purely imaginary solutions to \eqref{eqn:g1_eigenvalues}, for $\lambda = \pm \lambda_c(1):= \pm i (1.171)$ and 
 \begin{equation}
 \label{eqn:beta_c(1)}
 \beta = \beta_c(1) := 0.769;
 \end{equation}
 \item[(1.2)]  iterating the search around $(\beta_c(1),\lambda_c(1))$, the resulting complex eigenvalue goes from having a negative real part for $\beta < \beta_c(1)$ to a positive real part for $\beta > \beta_c(1)$;
 \item[(1.3)] no other purely immaginary solution $\lambda = \pm i x$ is found for $0 \leq x \leq 500$ and $0 \leq \beta \leq 20 $;
 \item[(1.4)] for $\beta < \beta_c(1)$ all the eigenvalues $\lambda = ix + y$ are such that $y < 0$. This was verified for $-100\leq x \leq 100$, $-100\leq y \leq 100$.
 \end{itemize}

\vspace{0.3cm}
\item Case $\gamma = 2$:
\vspace{0.15cm}
 \begin{itemize}
 \item[(2.1)] we find two conjugate purely imaginary solutions to \eqref{eqn:g2_eigenvalues}, for $\lambda = \pm\lambda_c(2) := \pm i(1.978)$ and 
 \begin{equation}
 \label{eqn:beta_c(2)}
 \beta = \beta_c(2):= 0.362;
 \end{equation}
 \item[(2.2)] analogous to (1.2);
 \item[(2.3)] analogous to (1.3), verified for $0 \leq x \leq 10$ and $0\leq \beta \leq 5$;
 \item[(2.4)] analogous to (1.4), verified for $-25 \leq x \leq 25$, $-25 \leq y \leq 25$;
 \item[(2.5)] apart from being sensibly slower, the numerical root finding for $\gamma = 2$ suffers from numerical instability issues. This is why we were able to check the results for much smaller intervals in this case.
 \end{itemize} 
\end{itemize}

\subsection{Finite particle system simulations}
\label{finite_particle}
We made several simulations of the particle system ($N$ large but finite, $N = 1500$) for $\gamma = 1, 2$, which seem in accordance with the above numerical results on the eigenvalues (compare with \eqref{eqn:beta_c(1)} and \eqref{eqn:beta_c(2)}).
This is a description of the evidences:
\begin{itemize}
\item For $\beta$ small the system is stable, in particular the magnetization goes to zero regardless of the initial datum (Figure \ref{stable_renewal}).
\item There is a critical $\beta$ (around $0.75$ for $\gamma = 1$, $0.35$ for $\gamma = 2$) above which the magnetization starts oscillating. Close to the critical points oscillations (Figure \ref{oscillating_renewal}) do not look very regular (corrupted by noise?), but they soon become very regular if $\beta$ is not too close to the critical value. We also made joint plots of the magnetization with the empirical mean of the $y_i$'s (Figure \ref{cycle_renewal}). A limit cycle seems to emerge.
\item As $\beta$ increases, the amplitude of the oscillation of the magnetization increases (Figure \ref{cyclebigger_renewal}), while the period looks nearly constant. As $\beta$ crosses another critical value (around $1.3$ for $\gamma = 1$, $1.65$ for $\gamma = 2$) oscillations disappear, and the sistem magnetizes, i.e.\! the magnetization stabilizes to a non-zero value, actually close to $\pm 1$ (Figure \ref{magnetized_renewal}).
\item The oscillations are lasting for a wider interval of $\beta$'s for $\gamma =2$ (from $\beta \approx 0.35$ until $\beta \approx 1.65$) than $\gamma = 1$ (from $\beta \approx 0.75$ until $\beta \approx 1.3$). The period is instead smaller for $\gamma = 2$ than for $\gamma = 1$.
\item For both $\gamma = 1,2$, the appearance of the oscillations does not seem to depend on the initial data for the dynamics, suggesting the possible presence of a \textit{global} Hopf bifurcation. 
\end{itemize}

\begin{figure}%
    \centering
    \subfloat[$\beta = 0.25$, $\gamma = 1$]{{\includegraphics[width=5.2cm]{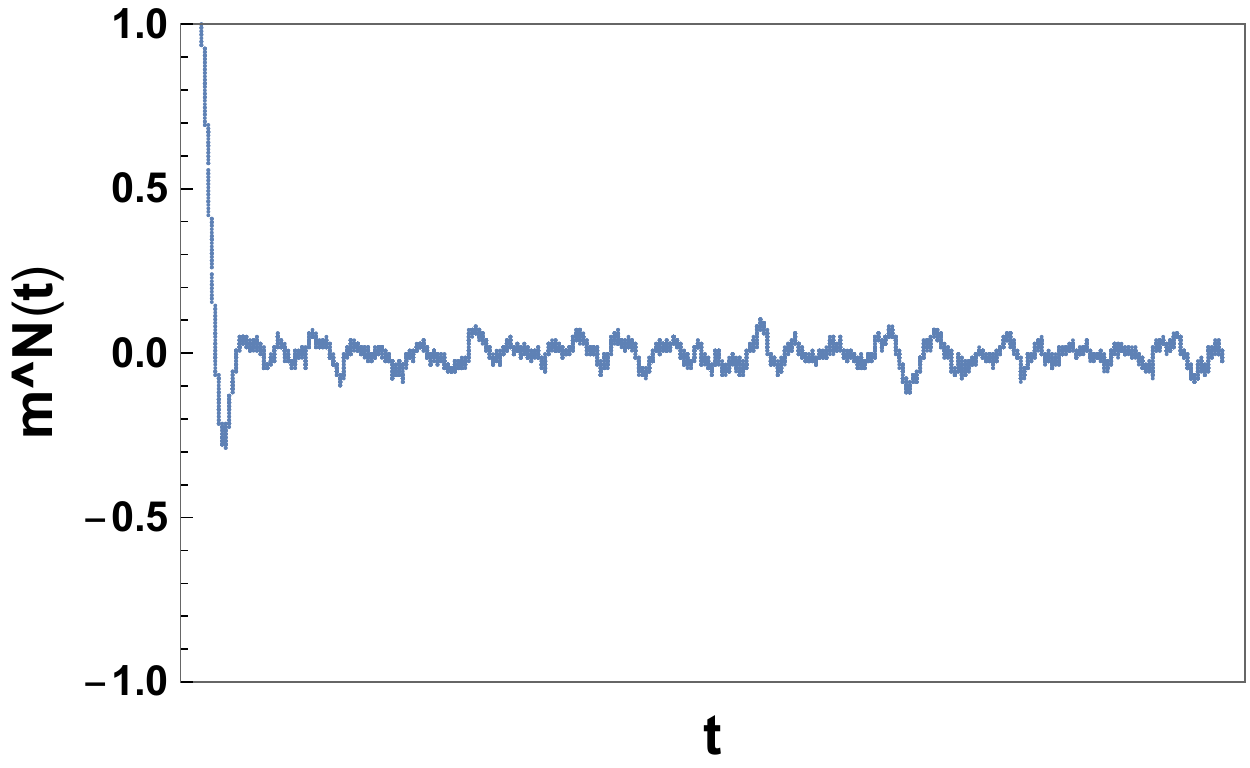} }}%
    \qquad
    \subfloat[$\beta = 0.1$, $\gamma = 2$]{{\includegraphics[width=5.2cm]{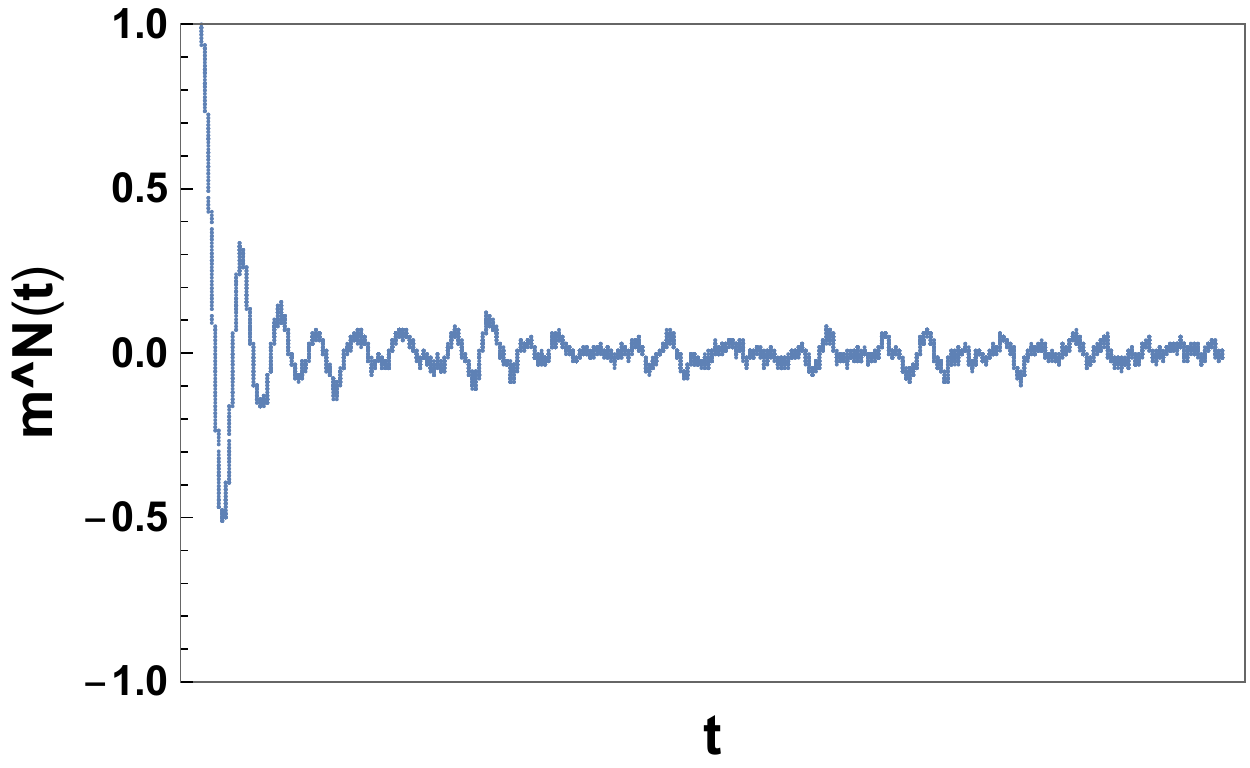} }}%
    \caption{Simulation of the finite particle system's dynamics for $\gamma=1$ (left) and $\gamma = 2$ (right), with number of spins $N =  1500$. We plot the empirical magnetization, with initial data $\sigma_i(0) = 1$ for every $i=1,\dots,N$.}   
    \label{stable_renewal}
\end{figure}

\begin{figure}%
    \centering
    \subfloat[$\beta = 0.75$, $\gamma = 1$]{{\includegraphics[width=5.2cm]{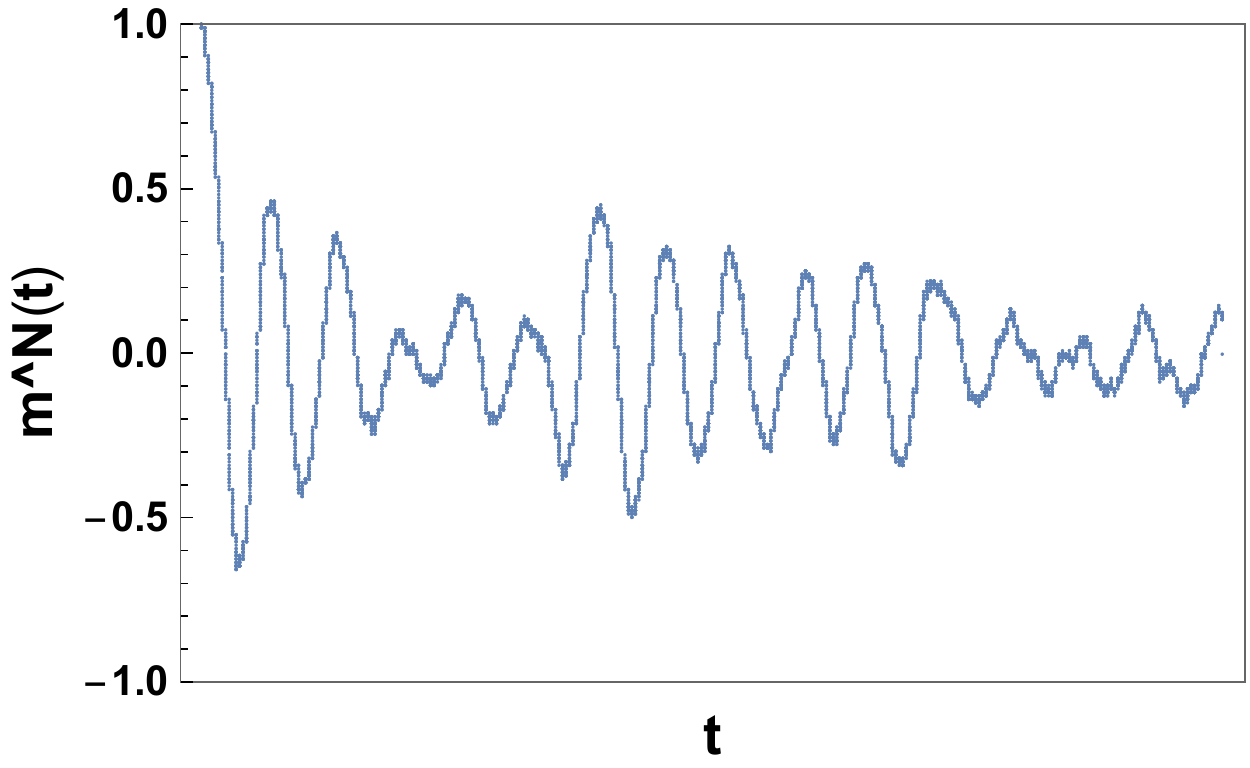} }}%
    \qquad
    \subfloat[$\beta = 0.35$, $\gamma = 2$]{{\includegraphics[width=5.2cm]{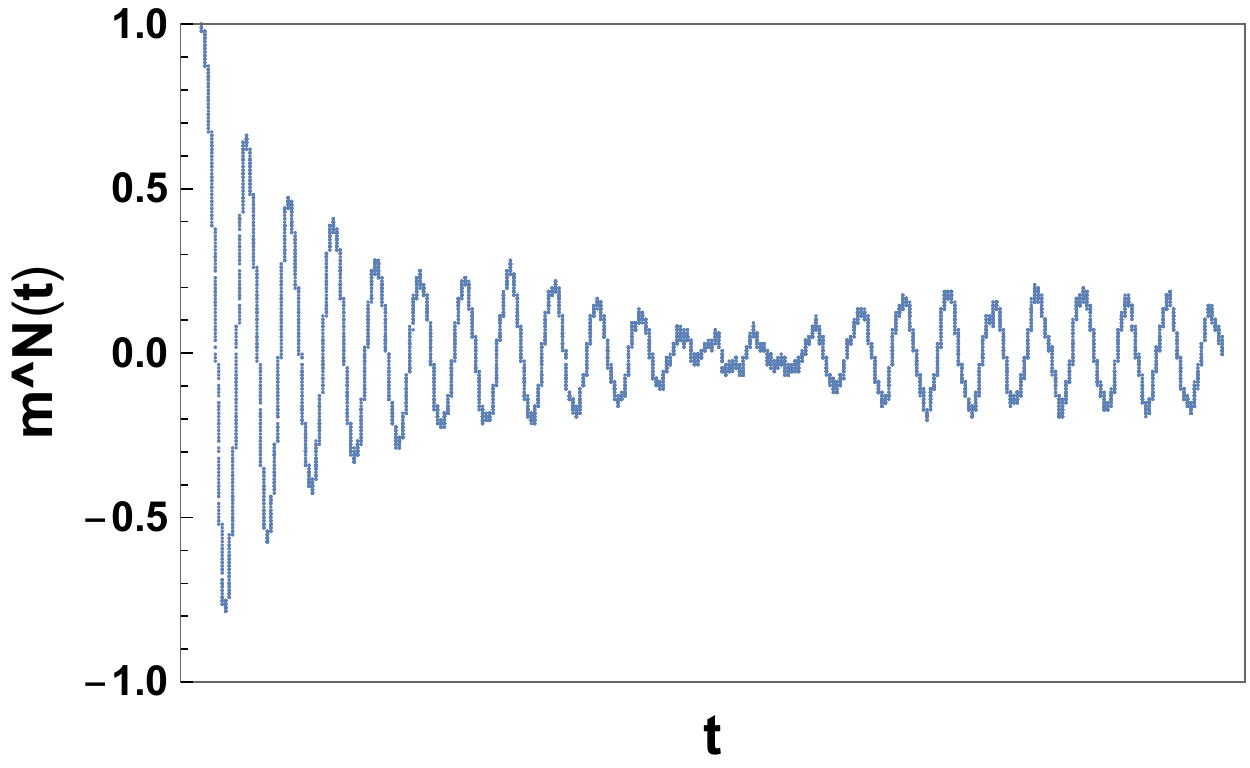} }}%
    \caption{Simulation of the finite particle system's dynamics for $\gamma=1$ (left) and $\gamma = 2$ (right), with number of spins $N =  1500$. We plot the empirical magnetization.}   
    \label{oscillating_renewal}
\end{figure}

\begin{figure}%
    \centering
    \subfloat[$\beta = 1.1$, $\gamma = 1$]{{\includegraphics[width=5.2cm]{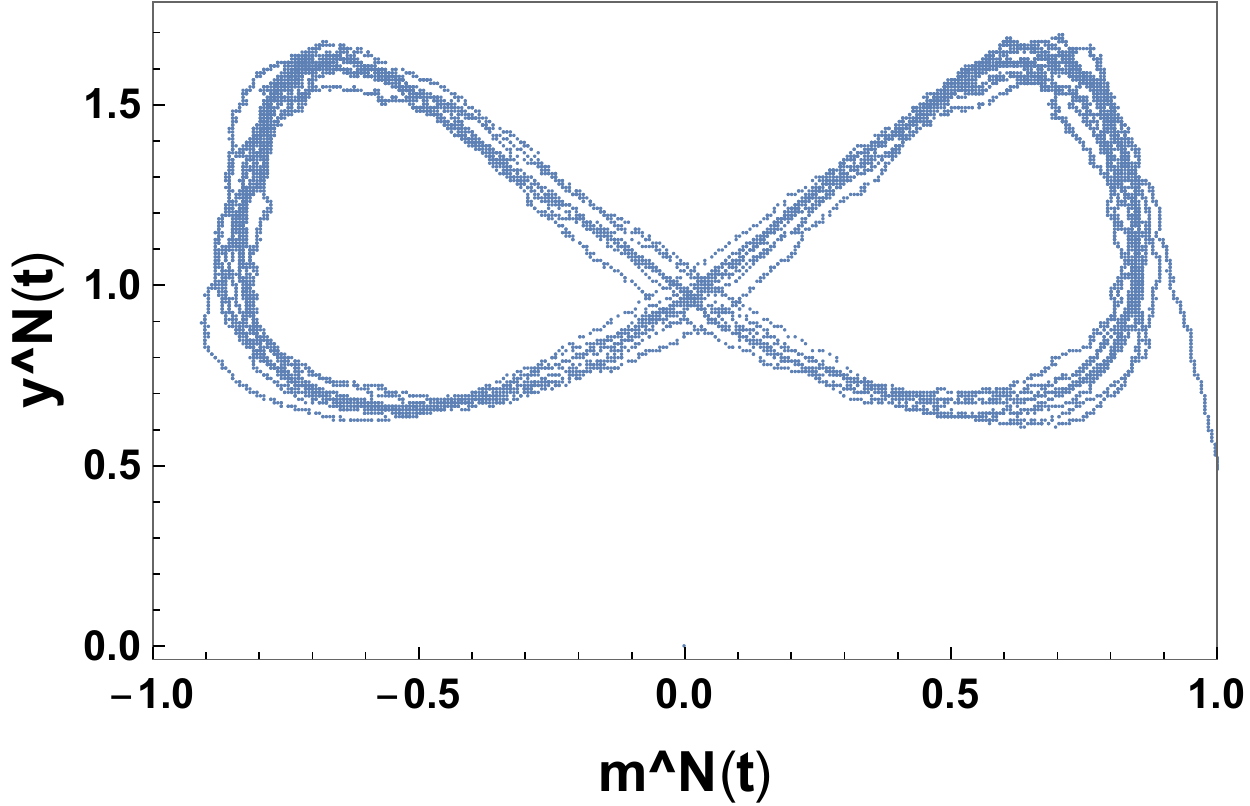} }}%
    \qquad
    \subfloat[$\beta = 0.6$, $\gamma = 2$]{{\includegraphics[width=5.2cm]{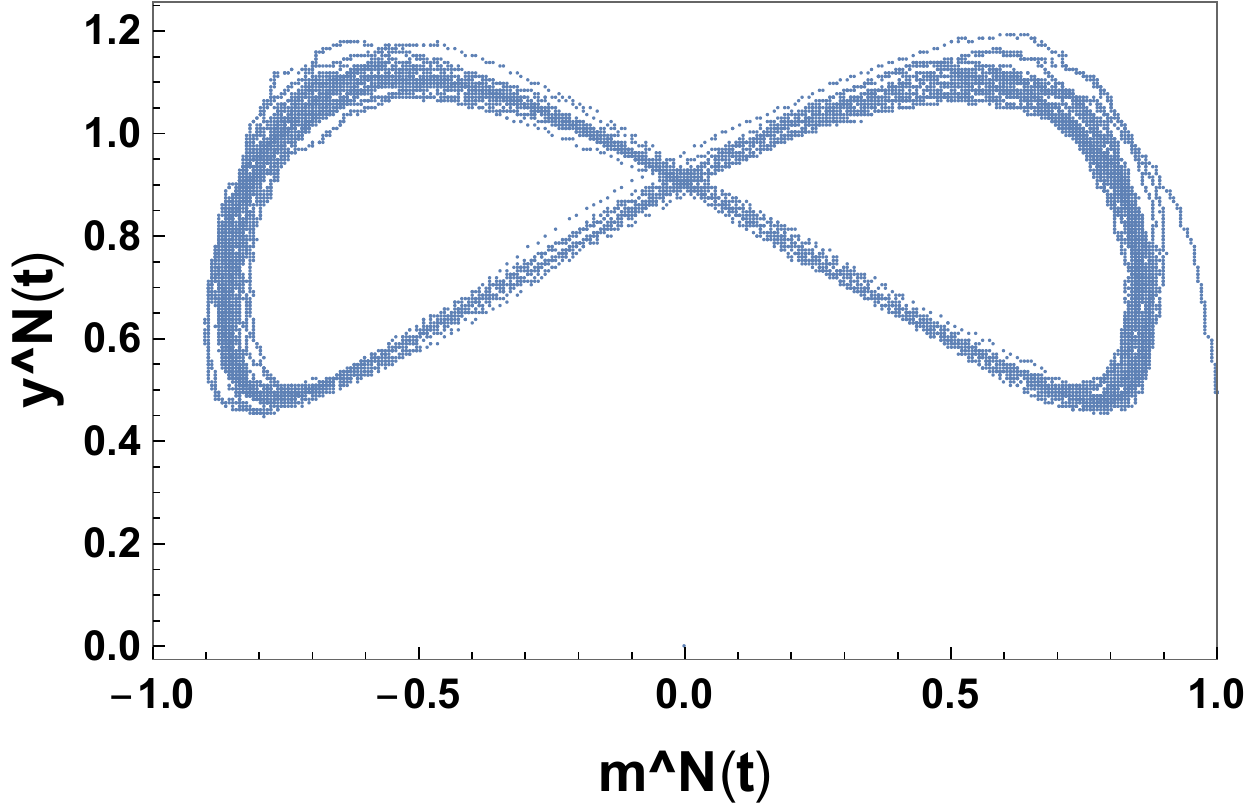} }}%
    \caption{Simulation of the finite particle system's dynamics for $\gamma=1$ (left) and $\gamma = 2$ (right), with number of spins $N =  1500$. We plot the empirical magnetization of the spins against the empirical mean of the $y_i$'s.}   
    \label{cycle_renewal}
\end{figure}

\begin{figure}%
    \centering
    \subfloat[$\beta = 1.2$, $\gamma = 1$]{{\includegraphics[width=5.2cm]{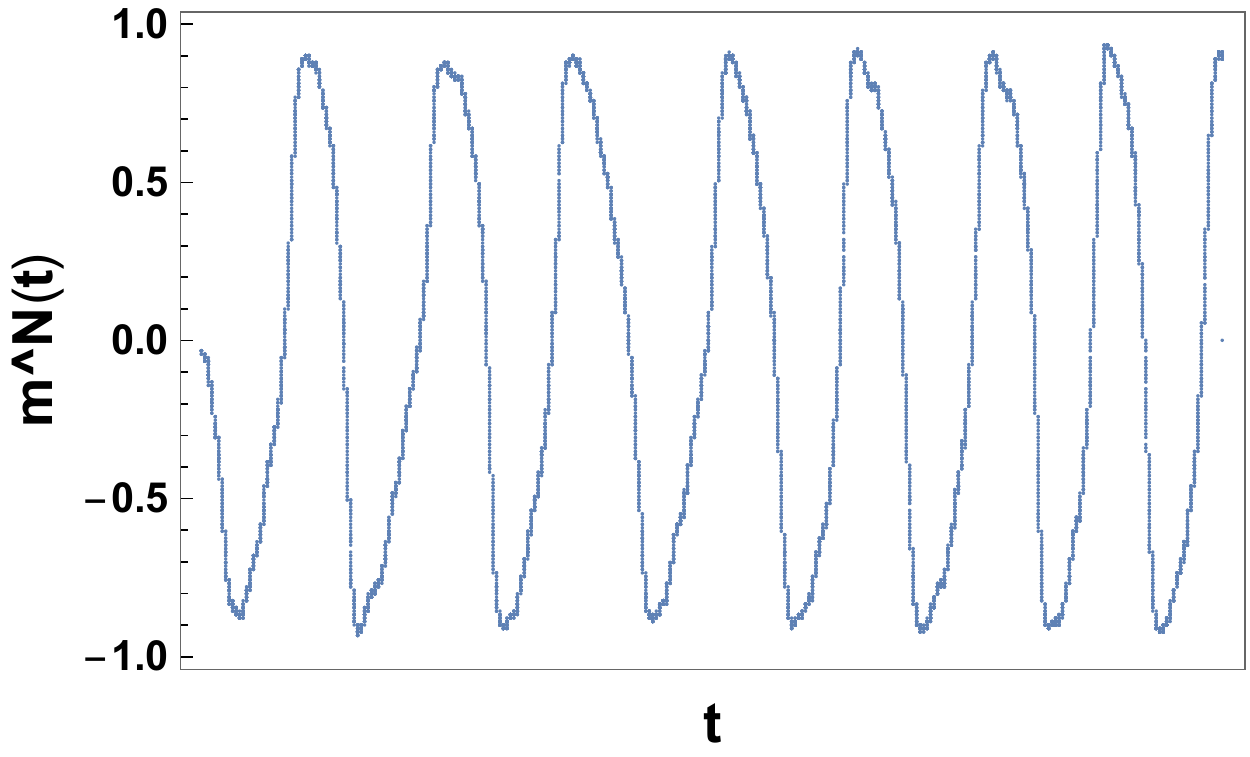} }}%
    \qquad
    \subfloat[$\beta = 0.9$, $\gamma = 2$]{{\includegraphics[width=5.2cm]{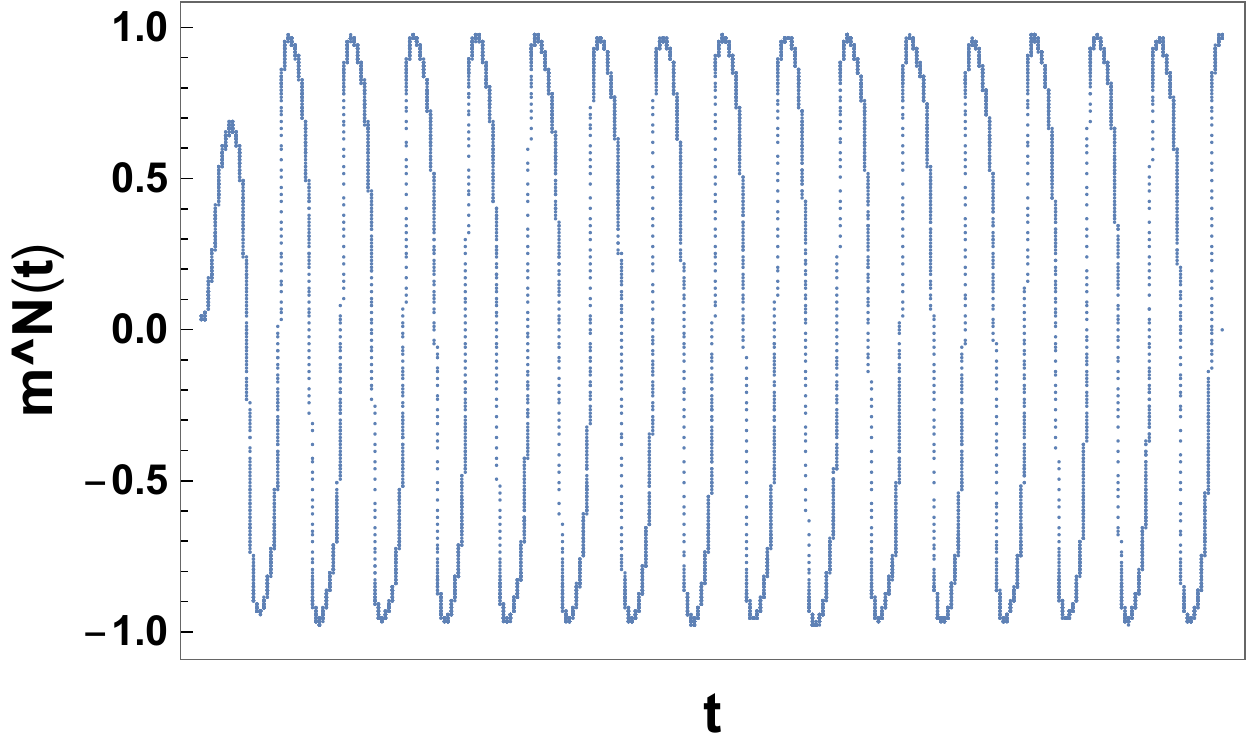} }}%
    \caption{Simulation of the finite particle system's dynamics for $\gamma=1$ (left) and $\gamma = 2$ (right), with number of spins $N =  1500$. We plot the empirical magnetization.}   
    \label{cyclebigger_renewal}
\end{figure}

\begin{figure}%
    \centering
    \subfloat[$\beta = 1.3$, $\gamma = 1$]{{\includegraphics[width=5.2cm]{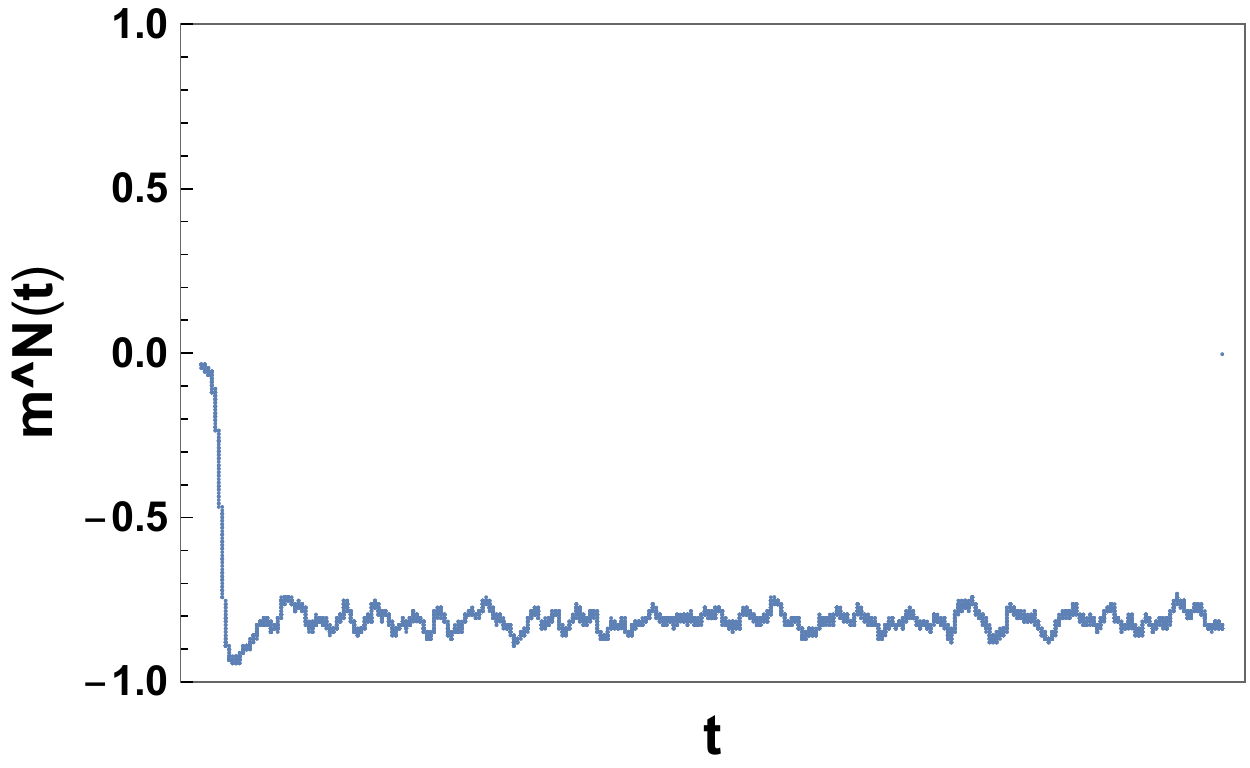} }}%
    \qquad
    \subfloat[$\beta = 1.65$, $\gamma = 2$]{{\includegraphics[width=5.2cm]{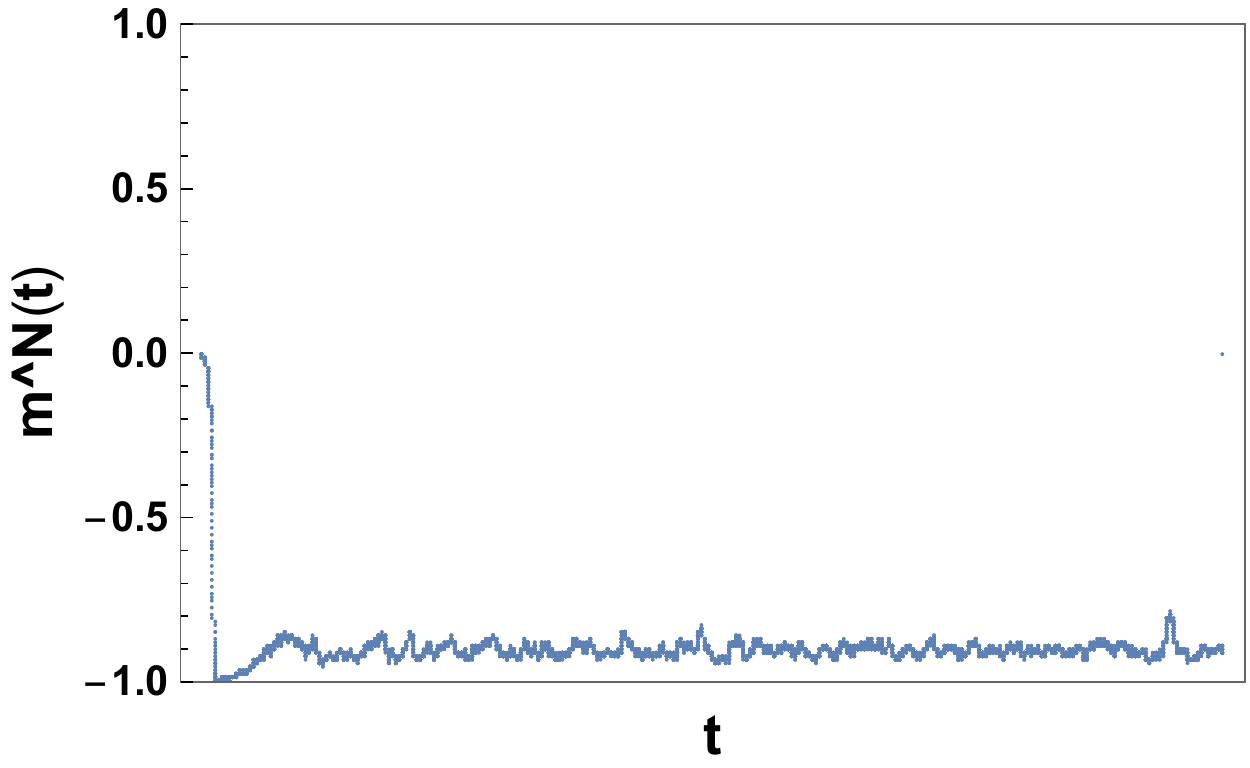} }}%
    \caption{Simulation of the finite particle system's dynamics for $\gamma=1$ (left) and $\gamma = 2$ (right), with number of spins $N =  1500$. We plot the empirical magnetization.}   
    \label{magnetized_renewal}
\end{figure}

\section{Proofs}
\label{proofs}
\subsection{Proof of Theorem \ref{prop_chaos_renewal}}
Here we prove rigorously a propagation of chaos property for the $N$-particle dynamics to its mean-field limit, for any $\gamma \in \mathbb{N}$. Actually, we establish the proofs for $\gamma = 1$, where the rates enjoy globally Lipschitz properties, and then we generalize them to any $\gamma\in \mathbb{N}$ in Remark \ref{gamma=2_chaos}. The generalization to non-Lipschitz rates is possible because of the a-priori bound on the variables $y_i$'s which, by definition, are such that $0 \leq y_i \leq T$, where $T < \infty$ is the final time horizon of the dynamics.
For the convenience of the reader, we write again the dynamics
\begin{equation}
\label{eqn:final_gamma_1}
\begin{cases}
(\sigma_i(t), y_i(t)) \mapsto (-\sigma_i(t),0),  \ \ \ \text{    with rate } \ \  y_i^{\gamma}(t) e^{-(\gamma + 1)\beta \sigma_i(t) m^N(t)}, \\
dy_i(t) = dt, \ \ \ \text{    otherwise,} 
\end{cases}
\end{equation}
and the mean-field version
\begin{equation}
\label{eqn:mf_final_1}
\begin{cases}
(\sigma(t), y(t)) \mapsto (-\sigma(t),0),  \ \ \ \text{    with rate } \ \ y^{\gamma}(t) e^{-(\gamma +1) \beta \sigma(t) m(t)}, \\
dy(t) = dt, \ \ \ \text{    otherwise,} 
\end{cases}
\end{equation}
with $m(t) = \mathbb{E}[\sigma(t)]$. 
We represent both the microscopic and the macroscopic model as solutions of certain stochastic differential equations driven by Poisson random measures, in order to apply the results in \cite{Graham}. As anticipated, in the proof we restrict to a finite interval of time $[0,T]$.

To begin with, let us fix a filtered probability space $\big((\Omega, \mathcal{F},\mathbb{P}),(\mathcal{F}_t)_{t \in [0,T]}\big)$ satisfying the usual hypotheses, rich enough to carry an inependent and identically distributed family $(\mathcal{N}_i)_{i \in \mathbb{N}}$ of stationary Poisson random measures $\mathcal{N}_i$ on $[0,T] \times \Xi$, with intensity measure $\nu$ on $\Xi := [0,+\infty)$ equal to the restriction of the Lebesgue measure onto $[0, +\infty)$.  
For any $N$, consider the system of Itô-Skorohod equations
\begin{equation}
\label{eqn:ito-sko}
\begin{cases}
\sigma_i(t) = \sigma_i(0) + \int_0^t \int_{\Xi}f_1(\sigma_i(s^-), \xi, m^N(s^-), y_i(s^-))\mathcal{N}_i(ds,d\xi),\\
y_i(t) = y_i(0) + t + \int_0^t \int_{\Xi}f_2(\sigma_i(s^-), \xi, m^N(s^-), y_i(s^-))\mathcal{N}_i(ds,d\xi),
\end{cases}
\end{equation}
and the corresponding limit non-linear \textit{reference particle}'s dynamics
\begin{equation}
\label{eqn:ito-sko_limit}
\begin{cases}
\sigma(t) = \sigma(0) + \int_0^t \int_{\Xi}f_1(\sigma(s^-), \xi, m(s^-), y(s^-))\mathcal{N}(ds,d\xi),\\
y(t) = y(0) + t + \int_0^t \int_{\Xi}f_2(\sigma(s^-), \xi, m(s^-), y(s^-))\mathcal{N}(ds,d\xi).
\end{cases}
\end{equation}
The functions $f_1, f_2 : \left\{-1,1\right\} \times \mathbb{R}^+ \times [-1,1] \times \mathbb{R}^+ \to \mathbb{R}$, modeling the jumps of the process, are given by
\begin{equation}
\label{eqn:jump_f12}
f_1(\sigma, \xi, m, y) := - 2\sigma \mathbbm{1}_{]0,\lambda[}(\xi), \quad f_2(\sigma, \xi,m,y):= -y\mathbbm{1}_{]0,\lambda[}(\xi), 
\end{equation}
with $\lambda := \lambda(\sigma,m,y)$ being the rate function $\lambda(\sigma,m,y) = y^{\gamma} e^{-(\gamma+1) \beta \sigma m}$.

\begin{prop}
\label{wp-itosko}
For $\gamma = 1$, Eqs.\! \eqref{eqn:ito-sko} and \eqref{eqn:ito-sko_limit} possess a unique strong solution for $t \in [0,T]$.
\end{prop}
\begin{proof}
With the choices in \eqref{eqn:jump_f12}, the well-posedeness of Eqs.\! \eqref{eqn:ito-sko} and \eqref{eqn:ito-sko_limit} follows by Theorems 1.2 and 2.1 in \cite{Graham}. Indeed, even though the function $f_2$ is not globally Lipschitz continuous in $y$, the $L^1$ Lipschitz assumption of the theorem still holds, by noting that 
\begin{align*}
\int_{\Xi} &\Big|f_2(\sigma,\xi,m,y) - f_2(\tilde{\sigma},\xi,\tilde{m},\tilde{y})\Big| d\xi = \int_{\Xi} \Big| y \mathbbm{1}_{]0,\lambda(\sigma,m,y)[}(\xi) - \tilde{y}\mathbbm{1}_{]0,\lambda(\tilde{\sigma},\tilde{m},\tilde{y})[}(\xi)\Big|d\xi\\
& \leq |y| \big|\lambda(\sigma,m,y) - \lambda(\tilde{\sigma},\tilde{m},\tilde{y})\big| + \big|\lambda(\tilde{\sigma},\tilde{m},\tilde{y})\big| \big|y - \tilde{y}\big|\\
&\leq |y| \big[\big|\lambda(\sigma,m,y) - \lambda(\tilde{\sigma},\tilde{m},\tilde{y})\big| + \big|\lambda(\tilde{\sigma},\tilde{m},\tilde{y})\big||y - \tilde{y}|\big] \leq CT \big[|m- \tilde{m}| + |y - \tilde{y}| + |\sigma - \tilde{\sigma}|\big],
\end{align*}
where in the last step we have used that, by construction, the processes $y_i(t) \leq T$ for every $t \in [0,T]$, so that the rates are a priori bounded and the Lipschitz properties of $y e^{-2\beta \sigma m}$ for $(y,\sigma,m) \in \mathbb{R}^+ \times\left\{-1,1\right\} \times[-1,1]$.
\end{proof}

Now, define the empirical measures $\mu^N := \frac{1}{N}\sum_{i=1}^N \delta_{(\sigma_i,y_i)}$, and their evaluation along the paths of \eqref{eqn:ito-sko}, 
\begin{equation}
\label{empirical_meas_sko}
\mu^N_t := \frac{1}{N}\sum_{i=1}^N \delta_{(\sigma_i(t), y_i(t))}.
\end{equation}
The measures $(\mu_t^N)_{t \in [0,T]}$ can be viewed as random variables with values in $\mathcal{P}(D)$, the space of probability measures on $D$, where $D:= \mathcal{D}\big([0,T];\left\{-1,1\right\}\times \mathbb{R}^+\big)$ is the space of $\left\{-1,1\right\}\times \mathbb{R}^+$-valued càdlàg functions equipped with the Skorohod topology.

\begin{proof}[Proof of Theorem \ref{prop_chaos_renewal}]
Consider the i.i.d.\! processes $(\tilde{\sigma}_i(t), \tilde{y}_i(t))_{i=1,\dots,N}$, coupled with the $N$-particle dynamics $(\sigma_i(t),y_i(t))_{i=1,\dots,N}$,
\begin{equation}
\label{eqn:ito-sko_iid_limit}
\begin{cases}
\tilde{\sigma}_i(t) = \tilde{\sigma}_i(0) + \int_0^t \int_{\Xi}f_1(\tilde{\sigma}_i(s^-), \xi, m(s^-), \tilde{y}_i(s^-))\mathcal{N}_i(ds,d\xi),\\
\tilde{y}_i(t) = \tilde{y}_i(0) + t + \int_0^t \int_{\Xi}f_2(\tilde{\sigma}_i(s^-), \xi, m(s^-), \tilde{y}_i(s^-))\mathcal{N}_i(ds,d\xi),
\end{cases}
\end{equation}
with $m(t) = \mathbb{E}[\tilde{\sigma}_i(t)]$. Let $(\tilde{\mu}_t^N)_{t \in [0,T]}$ be the empirical measure associated to \eqref{eqn:ito-sko_iid_limit}. Clearly, $(\tilde{\mu}_t^N)_{t \in [0,T]} \to (\mu_t)_{t \in [0,T]}$ in the weak convergence sense (by a functional LLN, see \cite{Graham} for e.g.). We are thus left to show
$$
\bm{d_1}\Big(\text{Law}\big((\mu_t^N)_{t \in [0,T]}\big),\text{Law}\big((\tilde{\mu}_t^N)_{t \in [0,T]}\big)\Big) \xrightarrow{N \to +\infty} 0,
$$
with $\bm{d_1}$ being the $1$-Wasserstein distance (which metrizes the weak convergence of probability measures) on $\mathcal{P}(\mathcal{P}(D))$. 
Since
$$
\bm{d_1}\Big(\text{Law}\big((\mu_t^N)_{t \in [0,T]}\big),\text{Law}\big((\tilde{\mu}_t^N)_{t \in [0,T]}\big)\Big) \leq \frac{1}{N}\sum_{i=1}^N \mathbb{E}\big[d_{Sko}\big((\sigma_i,y_i),(\tilde{\sigma}_i,\tilde{y}_i)\big)\big],
$$ 
with $d_{Sko}$ the Skorohod metric on $D$, it is enough to show that 
\begin{equation}
\label{eqn:crucial_renewal_chaos}
\frac{1}{N}\sum_{i=1}^N \mathbb{E}\Bigg[\sup_{t \in [0,T]} \Big(|\sigma_i(t) - \tilde{\sigma}_i(t)| + | y_i(t) - \tilde{y}_i(t)|\Big)\Bigg] \xrightarrow{N \to +\infty} 0.
\end{equation}
For the proof of \eqref{eqn:crucial_renewal_chaos}, we estimate, using the estimates of Proposition \ref{wp-itosko} for $f_2$,
\begin{align*}
\mathbb{E}&\left[\sup_{s \in [0,t]}|y_i(s) - \tilde{y}_i(s)|\right] \leq \mathbb{E}\Big[|y_i(0) - \tilde{y}_i(0)|\Big] \\
& \hspace{1cm}+ C \int_0^t \mathbb{E}\Bigg[|m^N(s) - m(s)| + |y_i(s) - \tilde{y}_i(s)| + |\sigma_i(s) - \tilde{\sigma}_i(s)|\Bigg]ds \\
& \leq C  \int_0^t \mathbb{E}\Bigg[\sup_{r \in [0,s]}|m^N(r) - m(r)| + \sup_{r \in [0,s]}|y_i(r) - \tilde{y}_i(r)| + \sup_{r \in [0,s]} |\sigma_i(r) - \tilde{\sigma}_i(r)|\Bigg]ds + C(N),
\end{align*}
with $C(N) \xrightarrow{N \to +\infty} 0$ because of the chaoticity assumption on the initial datum. Similarly for the $\sigma_i$'s, using the the Lipschitz continuity of $f_1$, we obtain
\begin{align*}
\mathbb{E}&\left[\sup_{s \in [0,t]}|\sigma_i(s) - \tilde{\sigma}_i(s)|\right] \leq \mathbb{E}\Big[|\sigma_i(0) - \tilde{\sigma}_i(0)|\Big] \\
& + C \int_0^t \mathbb{E}\Bigg[|m^N(s) - m(s)| + |y_i(s) - \tilde{y}_i(s)| + |\sigma_i(s) - \tilde{\sigma}_i(s)|\Bigg]ds \\
& \leq C  \int_0^t \mathbb{E}\Bigg[\sup_{r \in [0,s]}|m^N(r) - m(r)| + \sup_{r \in [0,s]}|y_i(r) - \tilde{y}_i(r)| + \sup_{r \in [0,s]} |\sigma_i(r) - \tilde{\sigma}_i(r)|\Bigg]ds+ C(N). 
\end{align*}
Denoting $\tilde{m}^N (t):= \frac{1}{N}\sum_{i=1}^N \tilde{\sigma}_i(t)$, we find
\begin{align*}
\mathbb{E}&\Bigg[\sup_{s \in [0,t]} |m^N(s) - m(s)|\Bigg] \leq \mathbb{E}\Bigg[\sup_{s \in [0,t]} |m^N(s) - \tilde{m}^N(s)|\Bigg] + \mathbb{E}\Bigg[\sup_{s \in [0,t]} |\tilde{m}^N(s) - m(s)|\Bigg]\\
&= \frac{1}{N}\sum_{j=1}^N \mathbb{E}\Bigg[\sup_{s \in [0,t]}|\sigma_j(s) - \tilde{\sigma}_j(s)|\Bigg] +  \mathbb{E}\Bigg[\sup_{s \in [0,t]} |\tilde{m}^N(s) - m(s)|\Bigg]\\
& = \mathbb{E}\Bigg[\sup_{s \in [0,t]}|\sigma_i(s) - \tilde{\sigma}_i(s)|\Bigg] + C(N),
\end{align*}
with $C(N) \xrightarrow{N \to +\infty} 0$ because of the chaoticity of the i.i.d.\! processes $(\tilde{\sigma}_i(t), \tilde{y}_i(t))_{i=1,\dots,N}$, and where in the equalities we have used the exchangeability of the processes $(\sigma_i, \tilde{\sigma}_i)_{i=1,\dots,N}$.
Recollecting the estimates, we have shown, for any $t \in [0,T]$,
\begin{align*}
\frac{1}{N}&\sum_{i=1}^N\left\{\mathbb{E}\Bigg[\sup_{s \in [0,t]} |\sigma_i(s) - \tilde{\sigma}_i(s)|\Bigg]+ \mathbb{E}\Bigg[\sup_{s\in [0,t]}| y_i(s) -\tilde{y}_i(s)|\Bigg]\right\}\\
& \leq C(N) + \int_0^t \frac{1}{N}\sum_{i=1}^N\mathbb{E}\Bigg[\sup_{r \in [0,s]} |\sigma_i(r) - \tilde{\sigma}_i(r)|+ \sup_{r \in [0,s]}|y_i(r) - \tilde{y}_i(r)| \Bigg]ds,
\end{align*}
which by the Gronwall's lemma applied to 
$$
\varphi(t) := \frac{1}{N}\sum_{i=1}^N\left\{ \mathbb{E}\Bigg[\sup_{s \in [0,t]} |\sigma_i(s) - \tilde{\sigma}_i(s)|\Bigg]+ \mathbb{E}\Bigg[\sup_{s\in [0,t]}| y_i(s) -\tilde{y}_i(s)|\Bigg]\right\},
$$
implies \eqref{eqn:crucial_renewal_chaos}, because $\varphi(T)$ is an upper bound for the left hand side of \eqref{eqn:crucial_renewal_chaos}.
\end{proof}

\begin{rem}
\label{gamma=2_chaos}
Proposition \ref{wp-itosko} and the proof of Theorem \ref{prop_chaos_renewal} can be generalized to any $\gamma \in \mathbb{N}$. Indeed, the same Lipschitz $L^1$ estimates on the rates of Proposition \ref{wp-itosko} (used also in Theorem \ref{prop_chaos_renewal}) hold by estimating
\begin{align*}
\big|\lambda&(\sigma,m,y) - \lambda(\tilde{\sigma},\tilde{m},\tilde{y})\big|= \big| y^\gamma e^{-(\gamma + 1)\beta m \sigma} - \tilde{y}^\gamma e^{-(\gamma + 1)\beta\tilde{m}\tilde{\sigma}}\big|\\
& \leq \big|y^\gamma e^{-(\gamma + 1)\beta m \sigma} - \tilde{y}^\gamma e^{-(\gamma + 1)\beta m\sigma} \big| + \big| \tilde{y}^\gamma e^{-(\gamma + 1)\beta m \sigma} - \tilde{y}^\gamma e^{-(\gamma + 1)\beta \tilde{m}\tilde{\sigma}}\big| \\
& \leq \big|e^{-(\gamma + 1)\beta m \sigma}\big| \big|y^\gamma - \tilde{y}^\gamma\big| + \tilde{y}^\gamma \big|e^{-(\gamma + 1)\beta m \sigma} - e^{-(\gamma + 1)\beta \tilde{m}\tilde{\sigma}}\big| \\
& \leq C\big|y - \tilde{y}\big| \big|p(y,\tilde{y})\big| + \tilde{y}^\gamma\Big[C |m - \tilde{m}| + C|\sigma - \tilde{\sigma}|\Big]\leq C\Big[|y - \tilde{y}| + |m - \tilde{m}| + |\sigma - \tilde{\sigma}|\Big],
\end{align*}
with $p(y,\tilde{y})$ a polynomial of degree $\gamma - 1$. In the last step we have used the a priori bounds on $y \leq T$ to get $|p(y,\tilde{y})| \leq C(T)$ and the Lipschitz properties of $e^{-(\gamma+1)\beta m \sigma}$ for $(\sigma,m) \in \left\{-1,1\right\} \times [-1,1]$. 
\end{rem}

\subsection{Proofs of the local analysis of Subsection \ref{local_fp}}
In this section we address the proofs of the results illustrated in Subsection \ref{local_fp}. We start with a remark on the derivation of the Fokker-Planck mean-field limit equation:
\begin{rem}
\label{equazione_bordo}
While the other equations in \eqref{eqn:kfp_renewal} are derived in a standard way from the expression of the generator \eqref{eqn:mf_limit_renewal}, the boundary integral condition might need to be motivated. In words, it is a mass-balance between the spins that have just jumped (thus having $y = 0$). 
We reason heuristically by discretizing the state space $[0,+\infty)$ in small intervals of amplitude $\varepsilon$. The discretized version of $y(t)$ takes values in $\left\{n\varepsilon \ : \ n \in \mathbb{N}\right\}$. The associated generator is, for $n \in \mathbb{N}$ and $\sigma \in \left\{-1,1\right\}$,
\begin{equation*}
\mathcal{L}_\varepsilon f(\sigma,n\varepsilon) = \frac{1}{\varepsilon}\big[f(\sigma,(n+1)\varepsilon) - f(\sigma,n\varepsilon)\big] + (n\varepsilon)^\gamma e^{-(\gamma+1)\beta \sigma m(t)}\big[f(-\sigma,0) - f(\sigma,n\varepsilon)\big].
\end{equation*}
Denoting $f(t,\sigma,0)$ the density of the discretized process in $(\sigma,0)$ at time $t$, it follows from the expression of $\mathcal{L}_\varepsilon$,
\begin{align*}
\frac{d}{dt}f(t,\sigma,0) = \sum_{n \in \mathbb{N}}(n\varepsilon)^\gamma e^{-(\gamma+1)\beta \sigma m(t)} f(t,-\sigma,n\varepsilon) - \frac{1}{\varepsilon}f(t,\sigma,0),
\end{align*}
that is the discretized version of the integral condition in \eqref{eqn:mf_limit_renewal}. 
\end{rem}
\begin{proof}[Proof of Proposition \ref{neutral_eq}]
Setting $m= 0$ in Sys.\! \eqref{eqn:kfp_renewal}, the stationary version of the first equation becomes
\begin{equation}
\label{eqn:stat_renewal}
 \frac{\partial}{\partial y} f(\sigma,y) + y^\gamma f(\sigma,y) = 0,
\end{equation}
whose solution is of the form $f^*(\sigma,y) = c(\sigma) f(\sigma,0) e^{-\frac{y^{\gamma+1}}{\gamma+1}}$. Denoting $\Lambda := \int_{0}^{+\infty} e^{-\frac{y^{\gamma+1}}{\gamma+1}}$, it is easy to see that the integral conditions imply $c(\sigma) = c(-\sigma) = \frac{1}{\Lambda}$ and $f(\sigma,0) = f(-\sigma,0) = \frac{1}{2}$. 
\end{proof}

\subsubsection{Formal derivation of Sys.\! \eqref{eqn:linear_renewal}}
\label{linearized_stationary}
We now compute formally the linearization of the operator associated to Sys.\! \eqref{eqn:kfp_renewal} around the solution \eqref{eqn:stat_sol} with $m = 0$.
Namely, if we write the first equation in \eqref{eqn:kfp_renewal} in operator form
$$
\frac{\partial}{\partial t} f(t,\sigma,y) - \mathcal{L}_{\gamma}^{nl}f(t,\sigma,y) = 0,
$$
with $\mathcal{L}_{\gamma}^{nl}f(t,\sigma,y) :=  -\frac{\partial}{\partial y} f(t,\sigma,y) - y^\gamma e^{-(\gamma + 1)\beta \sigma m(t)} f(t,\sigma,y)$, we want to find the linearized version of the operator $\mathcal{L}_{\gamma}^{nl}$.
 
For the purpose, we express a generic stationary solution to \eqref{eqn:kfp_renewal} as 
$$
f(\sigma,y) = f^*(\sigma,y) + \varepsilon g(\sigma,y),
$$ 
imposing 
\begin{equation}
\label{eqn:integral_lin_renewal}
\int_{0}^{\infty} [g(1,y) + g(-1,y)]dy = 0,
\end{equation}
so that $\int_{0}^{\infty} [f(1,y) + f(-1,y)]dy = 1$ is satisfied. 
We also denote $m_f :=  \int_{0}^{\infty} [f(1,y) - f(-1,y)]dy$, which by the above consideration satisfies 
\begin{equation}
\label{eqn:def_of_k_renewal}
m_f = 2\varepsilon\int_0^{\infty} g(1,y) dy =: \varepsilon k.
\end{equation}
The stationary version of the first equation in \eqref{eqn:kfp_renewal} becomes
$$
\frac{\partial}{\partial y} f^*(\sigma,y) + \varepsilon \frac{\partial}{\partial y} g(\sigma,y) + y^\gamma e^{-\beta \sigma \varepsilon k (\gamma + 1)}[f^*(\sigma,y) + \varepsilon g(\sigma,y)] = 0.
$$
By expanding at the first order in $\varepsilon$ the term $e^{-\beta \sigma \varepsilon k (\gamma + 1)} \approx 1 - (\gamma + 1)\beta \sigma \varepsilon k$, and by considering only the resulting linear terms in $\varepsilon$, we get
$$
\frac{\partial}{\partial y} f^*(\sigma,y) + \varepsilon \frac{\partial}{\partial y} g(\sigma,y) + y^\gamma f^*(\sigma,y) + y^\gamma \varepsilon g(\sigma,y) - y^\gamma(\gamma+1)\beta \sigma \varepsilon k f^*(\sigma,y) = 0.
$$
Finally, using that $f^*$ solves \eqref{eqn:stat_renewal} and substituting its expression \eqref{eqn:stat_sol}, we get
$$
\frac{\partial}{\partial y} g(\sigma,y) + y^\gamma g(\sigma,y) - \frac{\beta \sigma k(\gamma+1)}{2\Lambda} y^\gamma e^{-\frac{y^{\gamma + 1}}{\gamma + 1}} = 0.
$$
We can define the linearized operator as 
\begin{equation}
\label{eqn:lin_operator_renewal}
\mathcal{L}_{\gamma}^{\text{lin}}g(\sigma,y) :=  -\frac{\partial}{\partial y} g(\sigma,y) - y^\gamma g(\sigma,y) + \frac{\beta \sigma k(\gamma+1)}{2\Lambda} y^\gamma e^{-\frac{y^{\gamma + 1}}{\gamma + 1}}.
\end{equation}

We proceed with the linearization of the integral condition in the second line of Sys.\! \eqref{eqn:kfp_renewal}:
\begin{align*}
f^*(\sigma,0) + &\varepsilon g(\sigma,0) = \int_0^\infty[f^*(-\sigma,y) + \varepsilon g(-\sigma,y)] y^\gamma e^{\beta \sigma \varepsilon k (\gamma + 1)}\\
& \approx \int_0^\infty f^*(-\sigma,y) y^\gamma ( 1+ \beta \sigma \varepsilon k (\gamma + 1)) + \varepsilon \int_0^\infty g(-\sigma,y) y^\gamma ( 1+ \beta \sigma \varepsilon k (\gamma + 1))\\
& \approx \int_0^\infty f^*(-\sigma,y) y^\gamma +  \beta \sigma \varepsilon k (\gamma + 1)\int_0^\infty f^*(-\sigma,y) y^\gamma + \varepsilon \int_0^\infty g(-\sigma,y) y^\gamma.
\end{align*}
Using again that $f^*$ solves \eqref{eqn:stat_renewal} and its expression in \eqref{eqn:stat_sol}, we get
\begin{equation}
\label{eqn:border_lin_renewal}
g(\sigma,0) = \frac{\beta \sigma k (\gamma + 1)}{2\Lambda} + \int_0^\infty g(-\sigma,y) y^\gamma dy.
\end{equation}

In order to gain indications on the stability properties of the stationary solution to \eqref{eqn:kfp_renewal} with $m = 0$, we study the discrete spectrum of $\mathcal{L}_{\gamma}^{\text{lin}}$ defined in \eqref{eqn:lin_operator_renewal}, i.e., we search for the eigenfunctions $g$ and the eigenvalues $\lambda \in \mathbb{C}$, satisfying the linearized integral conditions \eqref{eqn:integral_lin_renewal} and \eqref{eqn:border_lin_renewal} found above, and such that
\begin{equation}
\label{eqn:eigen_renewal_operator}
\mathcal{L}_{\gamma}^{\text{lin}} g(\sigma,y) = \lambda g(\sigma,y),
\end{equation}
which is equivalent to 
\begin{equation}
\label{eqn:eigen_renewal}
\frac{\partial}{\partial y} g(\sigma,y) + y^\gamma g(\sigma,y) - \frac{\beta \sigma k(\gamma + 1)}{2\Lambda}y^\gamma  e^{-\frac{y^{\gamma +1}}{\gamma +1}} = -\lambda g(\sigma,y).
\end{equation}
The eigen-system around $m = 0$ is thus given by \eqref{eqn:linear_renewal}, where, recall by \eqref{eqn:def_of_k_renewal}, $k = 2\int_0^\infty g(1,y) dy$, and $\Lambda = \int_0^\infty e^{-\frac{y^{\gamma +1}}{\gamma + 1}} dy$.

\begin{rem}
\label{formal_derivation}
The derivation of the linearized operator \eqref{eqn:eigen_renewal_operator} was formal. One could think to define it more rigorously, by indicating an Hilbert space where $\mathcal{L}_{\gamma}^{\text{lin}}$ acts on. The natural choice appears to be (a subspace of) $\Big(L^2_{\mu_{\gamma}}\big(\mathbb{R}^+\big)\Big)^2$ satisfying conditions \eqref{eqn:integral_lin_renewal} and \eqref{eqn:border_lin_renewal},  where the outer square comes from the explicitation of the spin variable $\sigma = \pm 1$, and the measure $\mu_\gamma$ is defined as
\begin{equation}
\mu_\gamma(dy) := f^*(\sigma,y)dy = \frac{1}{2\Lambda}e^{-\frac{y^{\gamma + 1}}{\gamma +1}} dy.
\end{equation}
As in the computations we do not use the particular choice of domain of the operator or its properties, we do not investigate further on this. 
\end{rem}

\begin{proof}[Proof of Proposition \ref{gamma1_prop}]
In order to solve the first equation in \eqref{eqn:gamma_1}, we set $h(\sigma,y) := g(\sigma,y) e^{\frac{y^2}{2}}$. It holds
$$
\frac{\partial}{\partial y}h(\sigma,y) = -\lambda h(\sigma,y) + \frac{y \beta \sigma k}{\sqrt{\frac{\pi}{2}}},
$$
whose solution is
$$
h(\sigma,y) = e^{-\lambda y}\left[h(\sigma,0) + \frac{\beta \sigma k}{\sqrt{\frac{\pi}{2}}}\int_{0}^y u e^{\lambda u} du\right].
$$
Noting that $\int_0^y u e^{\lambda u} du = \frac{1}{\lambda^2} - \frac{e^{\lambda y}}{\lambda^2} + \frac{e^{\lambda y}}{\lambda}y$, we obtain
\begin{equation}
\label{eqn:g_aux_g1}
g(\sigma,y) = e^{-\frac{y^2}{2}} e^{-\lambda y} \left[ g(\sigma,0) + \frac{\beta \sigma k}{\sqrt{\frac{\pi}{2}}}\left(\frac{1}{\lambda^2} - \frac{e^{\lambda y}}{\lambda^2} + \frac{e^{\lambda y}}{\lambda}y\right) \right].
\end{equation}
We now impose the integral conditions. First, we note that $\int_0^\infty [g(\sigma,y) + g(-\sigma,y)]dy = 0$ is equivalent to $g(\sigma,y) + g(-\sigma,y) = 0$ for every $y \in \mathbb{R}^+$ because of expression \eqref{eqn:g_aux_g1}. For the computation of $k$, recalling notation \eqref{eqn:expression_h1g1}, we find
\begin{align*}
k = 2\int_0^\infty g(1,y) dy = 2g(1,0) H_1(\lambda) + 2\frac{\beta k}{\sqrt{\frac{\pi}{2}}}\frac{1}{\lambda^2} H_1(\lambda) - 2\frac{\beta k}{\lambda^2}\frac{1}{\sqrt{\frac{\pi}{2}}}\sqrt{\frac{\pi}{2}} + 2\frac{\beta k}{\sqrt{\frac{\pi}{2}}}\frac{1}{\lambda},
\end{align*}
so that
\begin{equation}
\label{eqn:value_of_k_g1}
k = \frac{2 g(1,0) H_1(\lambda)}{1 - 2\frac{\beta}{\lambda \sqrt{\frac{\pi}{2}}} - 2\frac{\beta H_1(\lambda)}{\lambda^2 \sqrt{\frac{\pi}{2}}} + 2\frac{\beta}{\lambda^2}}.
\end{equation}
The integral condition in the second line of \eqref{eqn:gamma_1} gives
\begin{align*}
g(\sigma,0) & = \frac{\beta \sigma k}{\sqrt{\frac{\pi}{2}}} + \int_0^\infty y\left[e^{-\frac{y^2}{2}}e^{-\lambda y}\left(g(-\sigma,0) - \frac{\beta \sigma k}{\sqrt{\frac{\pi}{2}}} \left(\frac{1}{\lambda^2} - \frac{e^{\lambda y}}{\lambda^2} + \frac{e^{\lambda y}}{\lambda} y\right)\right)\right]\\
& =  \frac{\beta \sigma k}{\sqrt{\frac{\pi}{2}}} - g(\sigma,0)(1-\lambda H_1(\lambda)) -  \frac{\beta \sigma k}{\sqrt{\frac{\pi}{2}}} \frac{(1-\lambda H_1(\lambda))}{\lambda^2} \\
&\hspace{2cm} + \frac{1}{\lambda^2} \frac{\beta \sigma k}{\sqrt{\frac{\pi}{2}}} - \frac{1}{\lambda} \frac{\beta \sigma k}{\sqrt{\frac{\pi}{2}}}\int_0^\infty y^2 e^{-\frac{y^2}{2}}dy \\
& = \frac{\beta \sigma k}{\sqrt{\frac{\pi}{2}}}  - g(\sigma,0)(1-\lambda H_1(\lambda)) - \frac{(1-\lambda H_1(\lambda))}{\lambda^2}\frac{\beta \sigma k}{\sqrt{\frac{\pi}{2}}}  + \frac{1}{\lambda^2}\frac{\beta \sigma k}{\sqrt{\frac{\pi}{2}}}  - \frac{1}{\lambda}\beta \sigma k.
\end{align*}
In the second equality we have used that $\int_0^\infty y e^{-\frac{y^2}{2}}e^{-\lambda y} = 1-\lambda H_1(\lambda)$ which can be obtained by an integration by parts.
Solving for $g(1,0)$ in the above
\begin{align*}
g(1,0)[2 - \lambda H_1(\lambda)] = \beta k \left[\frac{1}{\sqrt{\frac{\pi}{2}}} -  \frac{(1-\lambda H_1(\lambda))}{\lambda^2}\frac{1}{\sqrt{\frac{\pi}{2}}} +  \frac{1}{\lambda^2}\frac{1}{\sqrt{\frac{\pi}{2}}}  - \frac{1}{\lambda}\right].
\end{align*}
Substituting the value of $k$ we found in \eqref{eqn:value_of_k_g1}, we get
\begin{align*}
g(1,0)[2 &- \lambda H_1(\lambda)] = \frac{2 \beta g(1,0) H_1(\lambda)}{1 - 2\frac{\beta}{\lambda \sqrt{\frac{\pi}{2}}} - 2\frac{\beta H_1(\lambda)}{\lambda^2 \sqrt{\frac{\pi}{2}}} + 2\frac{\beta}{\lambda^2}}\left[\frac{1}{\sqrt{\frac{\pi}{2}}} -  \frac{(1-\lambda H_1(\lambda))}{\lambda^2}\frac{1}{\sqrt{\frac{\pi}{2}}} +  \frac{1}{\lambda^2}\frac{1}{\sqrt{\frac{\pi}{2}}}  - \frac{1}{\lambda}\right],
\end{align*}
which is equivalent to
\begin{equation}
\label{eqn:almost_final_g1}
2 - \lambda H_1(\lambda) = \frac{2\beta H_1(\lambda)\left[\lambda^2 + \lambda H_1(\lambda) - \lambda\sqrt{\frac{\pi}{2}}\right]}{\lambda^2\sqrt{\frac{\pi}{2}} - 2\beta \lambda - 2\beta H_1(\lambda) + 2\beta \sqrt{\frac{\pi}{2}}}.
\end{equation}
As a polynomial in $\lambda$, \eqref{eqn:almost_final_g1} can be written as
$$
-\lambda^3 H_1(\lambda)\sqrt{\frac{\pi}{2}} + \lambda^2\sqrt{2\pi} -4\beta\lambda - 4\beta H_1(\lambda) + 2\sqrt{2\pi}\beta = 0,
$$
or, grouping for $H_1(\lambda)$,
$$
H_1(\lambda)\left[-4\beta - \lambda^3 \sqrt{\frac{\pi}{2}}\right] + \sqrt{2\pi}\lambda^2 - 4\beta \lambda + 2\beta \sqrt{2\pi} = 0,
$$
i.e.\! the zeros of $H_{\beta,1}(\lambda)$, provided we prove expression \eqref{eqn:h1_g1} for $H_1(\lambda)$.
In fact, as defined in \eqref{eqn:expression_h1g1}, $H_1(\lambda)$ is a holomorphic function on $\mathbb{C}$, whose expression in series is
\begin{align*}
H_1(\lambda) & = \int_0^\infty e^{-\frac{y^2}{2}}e^{-\lambda y} dy = \sum_{n=0}^{\infty} (-1)^n \frac{\lambda^n}{n!}\int_0^\infty y^n e^{-\frac{y^2}{2}}dy.
\end{align*}
The latter integral is known 
\begin{equation}
\label{eqn:Gamma_f_g1}
\int_0^\infty y^n e^{-\frac{y^2}{2}}dy = 2^{\frac{1}{2}(n-1)}\Gamma\left(\frac{n+1}{2}\right),
\end{equation}
where $\Gamma(\cdot)$ is the Gamma function.
When $n = 2m + 1$, for the properties of the Gamma function on $\mathbb{N}$, \eqref{eqn:Gamma_f_g1} reduces to
$$
\int_0^\infty y^n e^{-\frac{y^2}{2}}dy = 2^{\frac{1}{2}(n-1)}\Gamma\left(\frac{n+1}{2}\right) = 2^m m!.
$$
For $n = 2m$ instead we have, by the property $\Gamma\left(l +\frac{1}{2}\right) = \frac{(2l -1)!!}{2^l}\sqrt\pi$ for any $l \in \mathbb{N}$, 
$$
\int_0^\infty y^n e^{-\frac{y^2}{2}}dy = 2^{\frac{1}{2}(n-1)}\Gamma\left(\frac{n+1}{2}\right) = \sqrt{\frac{\pi}{2}}(2m-1)!!.
$$
We use these equalities, and reorder the terms of the absolutely convergent series of $H_1(\lambda)$ to finally get
$$
H_1(\lambda) = \sqrt{\frac{\pi}{2}}\sum_{m=0}^{\infty}\frac{\lambda^{2m}}{(2m)!!} - \lambda \sum_{m=0}^{\infty} \frac{(2\lambda)^{2m} m!}{(2m+1)!}\frac{1}{2^m}.
$$
\end{proof}

\begin{proof}[Proof of Proposition \ref{gamma2_prop}]
We proceed as in the previous case, by setting $h(\sigma,y) := g(\sigma,y) e^{\frac{y^3}{3}}$, so that
$$
\frac{\partial}{\partial y} h(\sigma,y) = -\lambda h(\sigma,y) + \frac{3}{2\Lambda}\beta \sigma k y^2.
$$
Thus,
$$
h(\sigma,y) = e^{-\lambda y}\left[h(\sigma,0) + \frac{3\beta \sigma k}{2\Lambda}\int_0^y u^2 e^{\lambda u} du\right].
$$
Since $\int_0^y u^2 e^{\lambda u} du = \frac{1}{\lambda^3}[-2 + e^{\lambda y}(2 + \lambda y(-2 + \lambda y))]$, we can write
\begin{equation}
\label{eqn:g_aux_g2}
\begin{aligned}
g(\sigma,y) = e^{-\frac{y^3}{3}}e^{-\lambda y} \left[g(\sigma,0) + \frac{3\beta \sigma k}{2\Lambda}\frac{1}{\lambda^3}\left(- 2 + 2e^{\lambda y} -2\lambda y e^{\lambda y} + \lambda^2 y^2 e^{\lambda y}\right)\right].
\end{aligned}
\end{equation}
Recalling notation \eqref{eqn:expression_h2g2}, we compute
\begin{align*}
k & = 2\int_0^\infty g(1,y) dy = 2 H_2(\lambda) g(1,0) - 2H_2(\lambda) \frac{3 \beta k}{\Lambda \lambda^3} + 2\frac{3 \beta k}{\Lambda \lambda^3}\int_0^\infty e^{-\frac{y^3}{3}}dy \\
&- 2\frac{3\beta k}{\Lambda \lambda^2}\int_0^\infty y e^{-\frac{y^3}{3}}dy + \frac{3 \beta k}{\Lambda \lambda}\int_0^\infty y^2 e^{-\frac{y^3}{3}}\\
& = 2 H_2(\lambda) g(1,0) - 2H_2(\lambda) \frac{3 \beta k}{\Lambda \lambda^3}  + 2 \frac{3 \beta k}{\lambda^3} - 2 \frac{3\beta k}{\Lambda \lambda^2} \frac{\Gamma(2/3)}{3^{1/3}} + \frac{3\beta k}{\Lambda \lambda}, 
\end{align*}
which gives
\begin{equation}
\label{eqn:value_of_k_g2}
k = \frac{2 g(1,0) H_2(\lambda)}{1 + 2H_2(\lambda)\frac{3\beta}{\Lambda \lambda^3} - 2\frac{3\beta}{\lambda^3} + 2\frac{3\beta}{\Lambda \lambda^2}\frac{\Gamma(2/3)}{3^{1/3}} - \frac{3\beta}{\Lambda \lambda}}.
\end{equation}
As before, the condition $\int_0^\infty [g(\sigma,y) + g(-\sigma,y)]dy = 0$ in \eqref{eqn:gamma_2} is equivalent to $g(\sigma, y) + g(-\sigma, y) = 0$ for every $y \in \mathbb{R}^+$ because of  \eqref{eqn:g_aux_g2}. Using this observation for $y = 0$ in the other integral condition, we compute
\begin{align*}
g(\sigma,0) & = \frac{3 \beta \sigma k}{2\Lambda} + \int_0^\infty y^2 \Bigg[e^{-\frac{y^3}{3}}e^{-\lambda y}\Bigg(-g(\sigma,0)- \frac{3\beta \sigma k}{2\Lambda} \frac{1}{\lambda^3}(-2 + 2 e^{\lambda y} - 2\lambda y e^{\lambda y} + \lambda^2 y^2 e^{\lambda y})\Bigg)\Bigg]dy.
\end{align*} 
Observing that, by integration by parts, $\int_0^\infty y^2 e^{-\frac{y^3}{3}} e^{-\lambda y} dy = 1 -\lambda H_2(\lambda)$, we find
\begin{align*}
g(\sigma,0) &= \frac{3 \beta \sigma k}{2\Lambda} - (1 -\lambda H_2(\lambda))g(\sigma,0) + \frac{3\beta \sigma k}{\Lambda \lambda^3}(1-\lambda H_2(\lambda)) - \frac{3\beta \sigma k}{\Lambda \lambda^3} \\
& + \frac{3\beta \sigma k }{\Lambda \lambda^2}3^{1/3}\Gamma(4/3) - \frac{3\beta \sigma k}{2\Lambda \lambda}3^{2/3}\Gamma(5/3).
\end{align*}
Computing in $\sigma = 1$ and grouping for $g(1,0)$,
\begin{align*}
g(1,0)[2 - \lambda H_2(\lambda)] &= k\left[\frac{3 \beta}{2\Lambda} + \frac{3\beta}{\Lambda \lambda^3}(1-\lambda H_2(\lambda)) - \frac{3\beta}{\Lambda \lambda^3} + \frac{3\beta }{\Lambda \lambda^2}3^{1/3}\Gamma(4/3) - \frac{3\beta}{2\Lambda \lambda}3^{2/3}\Gamma(5/3)\right]\\
& = k \left[\frac{3 \beta}{2\Lambda} - \frac{3\beta}{\Lambda \lambda^2}H_2(\lambda) + \frac{3\beta }{\Lambda \lambda^2}3^{1/3}\Gamma(4/3) - \frac{3\beta}{2\Lambda \lambda}3^{2/3}\Gamma(5/3)\right].
\end{align*}
Plugging expression \eqref{eqn:value_of_k_g2} for $k$,
\begin{align*}
2 - \lambda H_2(\lambda) &= \frac{2 H_2(\lambda)}{1 + 2H_2(\lambda)\frac{3\beta}{\Lambda \lambda^3} - 2\frac{3\beta}{\lambda^3} + 2\frac{3\beta}{\Lambda \lambda^2}\frac{\Gamma(2/3)}{3^{1/3}} - \frac{3\beta}{\Lambda \lambda}}\left[\frac{3}{2\Lambda}\beta - \frac{3\beta}{\Lambda\lambda^2}H_2(\lambda) \right.\\
& \left.+ \frac{3\beta}{\Lambda\lambda^2}3^{1/3}\Gamma(4/3) - \frac{3\beta}{2\Lambda\lambda}3^{2/3}\Gamma(5/3)\right].
\end{align*}
This gives
\begin{align*}
2 - \lambda H_2(\lambda) = \frac{2\lambda H_2(\lambda)\left[\frac{3}{2}\beta \lambda^2 - 3\beta H_2(\lambda) + 3\beta 3^{1/3}\Gamma(4/3) - \frac{3}{2}\beta \lambda 3^{2/3}\Gamma(5/3)\right]}{\Lambda \lambda^3 + 6\beta H_2(\lambda) - 6\beta \Lambda + 6\beta\lambda\frac{\Gamma(2/3)}{3^{1/3}} - 3\beta \lambda^2},
\end{align*}
which is equivalent to
\begin{align*}
2 \Lambda \lambda^3 &+ 12 \beta H_2(\lambda) - 12\beta\Lambda + 12\beta\lambda\frac{\Gamma(2/3)}{3^{1/3}} - 6\beta \lambda^2 - \lambda^4 \Lambda H_2(\lambda) + 6\beta\lambda H_2(\lambda)\Lambda - 6\beta\lambda^2\frac{\Gamma(2/3)}{3^{1/3}} H_2(\lambda) \\
& = 6\beta\lambda H_2(\lambda)3^{1/3}\Gamma(4/3) - 3\beta\lambda^2 H_2(\lambda) 3^{2/3}\Gamma(5/3).
\end{align*}
As a polynomial in $\lambda$, this is
\begin{align*}
-\Lambda H_2(\lambda)& \lambda^4+ 2\Lambda\lambda^3 + \lambda^2\left[-6\beta -6\beta\frac{\Gamma(2/3)}{3^{1/3}}H_2(\lambda) + 3\beta H_2(\lambda) 3^{2/3}\Gamma(5/3)\right]\\
& +\lambda\left[6\beta\Lambda H_2(\lambda) -6\beta H_2(\lambda)3^{1/3}\Gamma(4/3) + 12\beta \frac{\Gamma(2/3)}{3^{1/3}}\right] + 12\beta H_2(\lambda) - 12\beta\Lambda = 0.
\end{align*}
Equivalently, in terms of $H_2(\lambda)$ we have
\begin{align*}
H_2(\lambda)&[12\beta - \lambda^4 \Lambda + 6\beta \lambda \Lambda - 6\beta \lambda 3^{1/3}\Gamma(4/3) + 3\beta \lambda^2 3^{2/3}\Gamma(5/3)- 6\beta\lambda^2\frac{\Gamma(2/3)}{3^{1/3}}]\\
& + \left[2\Lambda \lambda^3 - 12\beta\Lambda + 12\beta\frac{\Gamma(2/3)}{3^{1/3}}\lambda - 6\beta \lambda^2 \right] = 0,
\end{align*}
i.e.\! the zeros of $H_{\beta,2}(\lambda)$ in \eqref{eqn:holo_g2}, provided we show the validity of expression \eqref{eqn:h2_g2} for $H_2(\lambda)$.
As defined in \eqref{eqn:expression_h2g2}, $H_2(\lambda)$ is a holomorphic function on $\mathbb{C}$, which can be expressed in series as
\begin{align*}
H_2(\lambda) &= \int_0^\infty e^{-\lambda y}e^{-\frac{y^3}{3}}dy = \sum_{n=0}^{\infty} (-1)^n \frac{\lambda^n}{n!}\int_0^\infty y^n e^{-\frac{y^3}{3}}dy\\
& = \sum_{n=0}^{\infty} (-1)^n \frac{\lambda^n}{n!} 3^{\frac{1}{3}(n-2)}\Gamma\left(\frac{n+1}{3}\right),
\end{align*}
which is expression \eqref{eqn:h2_g2}, where we have used the formula for $\int_0^\infty y^n e^{-\frac{y^3}{3}}dy = 3^{\frac{1}{3}(n-2)}\Gamma\left(\frac{n+1}{3}\right)$.
\end{proof}

\end{document}